\documentclass[11pt]{article}
\usepackage{latexsym}
\usepackage{eucal}
\usepackage{times}
\usepackage{geometry}
\usepackage{amsmath,amssymb,amsthm, amscd,amsfonts,stackrel,mathrsfs,color}
\usepackage{mathrsfs}
\usepackage[all]{xy}
\usepackage{mathrsfs}
\usepackage{graphicx}

\newtheorem{proposition}{Proposition}
\newtheorem{theorem}{Theorem}
\newtheorem{remark}{Remark}
\newtheorem{lemma}{Lemma}

\newtheorem{corollary}{Corollary}
\newtheorem{definition}{Definition}

\newcommand{\dist}{\mathrm{dist}}
\newcommand{\II}{\mathrm{II}}
\newcommand{\ra}{\rightarrow}
\newcommand{\di}{\mathrm{d}}
\newcommand{\disp}{\displaystyle}
\newcommand{\cn}{\mathrm{cn}}
\newcommand{\sn}{\mathrm{sn}}

\newcommand{\eps}{\varepsilon}
\newcommand{\R}{\mathbb{R}}
\newcommand{\Ric}{\mathrm{Ric}}
\newcommand{\Sect}{\mathrm{Sect}}
\newcommand{\Tr}{\mathrm{Tr}}
\newcommand{\mcal}{\mathcal}
\newcommand{\metric}{\langle \, , \, \rangle}
\newcommand{\HH}{\mathbb{H}}
\newcommand{\cut}{\mathrm{cut}}
\newcommand{\lip}{\mathrm{Lip}}

\newcommand{\loc}{\mathrm{loc}}
\newcommand{\PP}{\mathcal{P}}

\newcommand{\measrest}{%
  \,\raisebox{-.127ex}{\reflectbox{\rotatebox[origin=br]{-90}{$\lnot$}}}\,%
}

\begin{document}

\title{A barrier principle at infinity for varifolds with bounded mean curvature}
\author{J. H. Lira\thanks{Partially supported by CNPq/FUNCAP/PRONEX Grant $\#$ PR2-0054-00009.01.00/11 } \and A. A. Medeiros\thanks{Partially supported by CNPq} \and L. Mari \and E.S. Gama}
\maketitle

\begin{abstract}
Our work investigates varifolds $\Sigma \subset M$ in a Riemannian manifold, with arbitrary codimension and bounded mean curvature, contained in an open domain $\Omega$. Under mild assumptions on the curvatures of $M$ and on $\partial \Omega$, also allowing for certain singularities of $\partial \Omega$, we prove a barrier principle at infinity, namely we show that the distance of $\Sigma$ to $\partial \Omega$ is attained on $\partial \Sigma$. Our theorem is a consequence of sharp maximum principles at infinity on varifolds, of independent interest.
\end{abstract}

\tableofcontents

\section{Introduction}

The classical barrier principle, also called tangency principle, states that a connected, minimal hypersurface $\Sigma^{m-1} \ra M^m$, with image inside a mean convex set $\Omega$, cannot touch $\partial \Omega$ unless $\Sigma \subset \partial \Omega$ (for a short proof, see \cite{Esch_1}). Similarly, and because of the unique continuation principle, two connected minimal hypersurfaces $\Sigma_1,\Sigma_2$ that touch at a common point $p$, in a way that $\Sigma_1$ locally lies on one side of $\Sigma_2$ near $p$, must coincide. To the best of our knowledge, extensions of the barrier principle to higher codimensional submanifolds $\Sigma^\ell \ra M^m$ were first obtained in \cite{JT} by L.P. Jorge and F. Tomi: if a connected $\ell$-dimensional minimal submanifold $\Sigma^\ell$ lies in a subset $\Omega \subset M$ whose boundary is $\ell$-mean convex, then $\Sigma$ is disjont from $\partial \Omega$ or it is contained in $\partial \Omega$. Here, $\ell$-mean convexity means that the sum of the smallest $\ell$ principal curvatures of the second fundamental form $\II_{\partial \Omega}$ in the inward direction is non-negative, equivalently, that the trace of $\II_{\partial \Omega}$ on $\ell$-dimensional subspaces is non-negative. One of the typical examples of barrier in Euclidean space $\mathbb{R}^{m}=\mathbb{R}^{m-s}\times\mathbb{R}^s$ is the cone $u^{-1}(0)$, where 
\begin{equation}\label{def_Phi_DS}
u : = \varrho^2_{\mathbb{R}^{m-s}} - c \varrho^2_{\mathbb{R}^s},
\end{equation}
$c$ is a positive constant and $\varrho_{\mathbb{R}^ {m-s}}$ and $\varrho_{\mathbb{R}^ {s}}$ are distance functions relative to fixed reference points in $\mathbb{R}^{m-s}$ and $\mathbb{R}^s$, respectively. If $s < \ell$ and $c \le \frac{\ell-s}{s}$, then the boundary of $\Omega = u^{-1}((-\infty, 0))$ is $\ell$-mean convex in the inward direction. As a matter of fact, under the same bound the restriction of $u$ to a minimal submanifold $\Sigma^\ell \ra \Omega$ is subharmonic; therefore, using the classical maximum principle, U. Dierkes and D. Schwab \cite{Dierkes,DS} proved some enclosure as well as nonexistence theorems, both for minimal and for submanifolds with bounded mean curvature.\par
The need to establish a barrier principle for possibly nonsmooth submanifolds has stimulated various interesting works in recent years, and by now a sharp theorem in full generality (i.e. allowing singularities for both $\Sigma$ and $\partial \Omega$) is only available for codimension $1$, singular minimal (stationary) hypersurfaces, cf. \cite{wick}. Indeed, the presence of singularities makes the problem quite more delicate, in view of the possibility that the two hypersurfaces touch at a point that is singular for both. However, when $\partial \Omega$ is smooth more is known, even for higher codimensional submanifolds $\Sigma$. We list the main achievements.
\begin{itemize}
\item[-] If $\partial \Omega$ is smooth, then 
\begin{itemize}
\item[-]  B. Solomon and B. White \cite{solowhi} proved the barrier principle when $\Sigma$ is an $(m-1)$-dimensional varifold that is stationary for an even, parametric elliptic functional $F$, including the area one, under the appropriately defined mean convexity of $\partial \Omega$ with respect to $F$; 
\item[-] for the area functional, the result in \cite{solowhi} was later extended to $\ell$-dimensional varifolds $\Sigma$ by B. White in \cite{White}, see also 
 \cite[Thms 7.1 and 7.3]{white_2}; 
\item[-] in a different direction, \cite{solowhi} has also been extended to $(m-1)$-varifolds with mean curvature ${\bf H} \in L^p_\loc$ with $p>m-1$, by R. Sch\"atzle in \cite{schatzle};
\item[-] recent results, for branched surfaces with bounded mean curvature and for energy stationary currents, can be found in P. Henkemeyer's \cite{henke,henke2}.
\end{itemize}
\item[-] If both $\partial \Omega$ and $\Sigma$ are possibly singular, the barrier principle has been proved:
\begin{itemize}
\item[-] by M. Moschen \cite{moschen}, when $\partial \Omega$ and $\Sigma$ are oriented, area minimizing boundaries; 
\item[-] by L. Simon \cite{simon}, when $\partial \Omega$ and $\Sigma$ are integer area minimizing $(m-1)$-currents;
\item[-] by T. Ilmanen \cite{ilmanen}, when $\partial \Omega$ and $\Sigma$ are stationary, integral $(m-1)$-varifolds whose singular set has locally finite $(m-2)$-dimensional measure, and eventually 
\item[-] by N. Wickramasekera \cite{wick}, when $\partial \Omega$ and $\Sigma$ are stationary, integral $(m-1)$-varifolds and $\mathscr{H}^{m-1}(\mathrm{sing} \, \partial \Omega) = 0$. He also showed that the result is best possible.
\end{itemize}

\end{itemize}

The main goal of the present paper is to establish a barrier principle \emph{at infinity} for varifolds of arbitrary codimension inside $\Omega$. Namely, if $\Sigma$ is a $\ell$-dimensional varifold (not necessary rectifiable) with support inside $\Omega$, under suitable mean curvature conditions on $\Sigma$ related to those of $\partial \Omega$, and under suitable curvature bounds on the underlying manifold $M$ (now necessary,  since the problem is not local any more), we aim to prove that
	\[
	\mathrm{dist}\big( {\rm spt } \|\Sigma\|, \partial \Omega\big) = \mathrm{dist}\big( {\rm spt } \|\partial \Sigma\|, \partial \Omega\big).
	\]
Results of this type have been investigated in recent years for minimal surfaces in $\R^3$, see the works of W.H. Meeks and H. Rosenberg \cite{meeksrosenberg} and A. Alarc\'on, B. Drinovec Drnov\u{s}ek, F. Forstner\u{i}c, and F.J. L\'opez \cite{adfl}, that will be compared to our main result in due course in the paper. Our theorem is also tightly related to the general half-space theorem proved by L. Mazet in \cite{mazet}. To accomplish our goal, we shall guarantee the validity of suitable forms at infinity of the maximum principle. Natural conditions to be put on $\Sigma$ are of potential-theoretic and stochastic nature, precisely we need some suitable form of parabolicity, stochastic completeness or of the validity of Omori-Yau principles on $\Sigma$, see \cite{amr,PRS1} for a detailed account. While the full Omori-Yau property, to present, requires a structure on the underlying manifold $\Sigma$ that is richer than merely being a $\ell$-dimensional varifold with mean curvature in $L^\infty$ (cf. \cite{maripessoa}), the refined integral estimates developed in \cite{prs_gafa,PRS1} by S. Pigola, M. Rigoli and A.G. Setti to ensure the weak maximum principle at infinity (equivalent, in the smooth setting, to the stochastic completeness of $\Sigma$) are very well suited to be adapted to the varifold setting.
%
%
Hereafter, given a $2$-covariant tensor $A$ with eigenvalues $\lambda_1(A) \le \lambda_2(A) \le \ldots \le \lambda_m(A)$, and $\ell \in \{1, \ldots, m\}$, we set
$$
\PP_\ell^-[A] : = \frac{\lambda_1(A) + \ldots + \lambda_\ell(A)}{\ell}, 
$$
Our main analytical results, Theorems \ref{MPV} and \ref{MPV_para} below, are inspired by \cite{prs_gafa,PRS1} and are tightly related to a recent work of B. White \cite{white_2}. For $0 \le h \in C(M)$, the author in \cite{white_2} defined a $(\ell,h)$-set $A \subset M$ to be a closed subset such that the following holds: whenever $u \in C^2(M)$ is such that $u_{|A}$ has a local maximum at $x \in A$, 
	\[
	\PP_\ell^-[\nabla^2 u](x) - h|\nabla u|(x) \le 0
	\]
(note that the function $h$ we use here corresponds to $h/\ell$ in White's definition). He showed that $\ell$-dimensional varifolds with normalized (generalized) mean curvature ${\bf H}$ satisfying $\|{\bf H}\|_\infty \le h$ are $(\ell,h)$ sets and, remarkably, this is so also for the blow-up set of sequences of such varifolds. 
%
%
%
%
We investigate versions at infinity of the above property, and we prove, respectively, a maximum principle at infinity (Theorem \ref{MPV}) and a parabolicity criterion (Theorem \ref{MPV_para}), currently restricted to the varifold setting. These results are used to establish the main barrier principle in the present paper: to state the theorem, recall that, for $\ell \in \{1,\ldots, m-1\}$, the $\ell$-th (normalized) Ricci curvature is the function
$$
v \in T_xM \quad \longmapsto \quad \Ric^{(\ell)}(v) : =  \inf_{\footnotesize{\begin{array}{c}
\mcal{W} \le v^\perp \\
\dim \mcal{W} = \ell
\end{array}}
} \left( \frac{1}{\ell} \sum_{j=1}^\ell \mathrm{Sect}(v \wedge e_j)\right), 
$$
where $\{e_j\}$ is an orthonormal basis of $\mcal W$. Hereafter, given $c \in \R$, with $\Ric^{(\ell)} \ge -c$ we shortly mean the inequality
$$
\Ric^{(\ell)}(v_x) \ge - c \qquad \forall\, x \in M, \ v_x \in T_xM.
$$
The function $\Ric^{(\ell)}$ interpolates between the sectional $(\ell=1)$ and Ricci $(\ell = m-1)$ curvatures, and with our chosen normalization the following implications are immediate:
$$
\Sect \ge -c \ \  \Longrightarrow \ \ \Ric^{(\ell-1)} \ge - c \ \ \Longrightarrow \ \ \Ric^{(\ell)} \ge - c \ \  \Longrightarrow \ \ \Ric \ge -c,
$$

%


The basic notions of varifold theory that we need are collected in Section \ref{sec_maxprinc}. We just observe that, if $B_r \subset M$ is a geodesic ball and $\Sigma$ is the varifold associated to a smooth $\ell$-dimensional submanifold of $M$, the quantity $\|\Sigma\|(B_r)$ coincides with the $\ell$-dimensional measure  of $\Sigma \cap B_r$. Also, we emphasize that the mean curvature vector ${\bf H}$ is assumed to be normalized, that is, in a smooth setting its value in a normal direction is the average of the principal curvatures and not their sum.

\begin{theorem}\label{omori-yau_vari}
Let $(M^m,\metric)$ be a complete manifold satisfying 
\begin{equation}\label{riccik_vari}
\Ric^{(\ell-1)} \ge - c, 
\end{equation}
for some $\ell \in \{2,\ldots, m-1\}$ and some $c \in \R$. Let $\Omega \subset M$ be an open set whose second fundamental form $\II_{\partial \Omega}$ in the inward direction satisfies  
\begin{equation}\label{finiteform_vari}
\PP_{\ell-1}^-[\II_{\partial \Omega}] \ge \Lambda_{\ell-1} \ \ \qquad \PP_\ell^-[\II_{\partial \Omega}] \ge \Lambda_\ell \ge 0
\end{equation}
in the barrier sense, for some constants $\Lambda_{\ell-1} \in \R$, $\Lambda_\ell \in [0, \infty)$, and that has locally bounded bending from outwards. Consider a $\ell$-dimensional varifold $\Sigma$ with connected support and a (possibly nonzero) generalized boundary, satisfying 
\begin{equation}\label{condi_volume}
\liminf_{r \ra \infty} \frac{\log \|\Sigma\|(B_r)}{r^2} < \infty,
\end{equation}
and with normalized (generalized) mean curvature $\mathbf{H} \in L^\infty(M,\|\Sigma\|)$. 
\begin{itemize}
\item[$(\mathbb{A})$] If $\|{\bf H}\|_\infty < \Lambda_\ell$ and $\Lambda_\ell^2 \ge c$, then   
\begin{equation}\label{bonita}
{\rm dist}\big({\rm spt } \|\Sigma\|, \partial \Omega\big) = \dist \big({\rm spt } \|\partial \Sigma\|, \partial \Omega \big).
\end{equation}
Moreover, $\Omega$ does not contain $\ell$-dimensional varifolds $\Sigma$ satisfying the above assumptions and with no boundary.
\item[$(\mathbb{B})$] If $\|{\bf H}\|_\infty < \Lambda_\ell$ and $\Lambda_\ell^2 < c$, then
\begin{equation}\label{bonita_bis}
{\rm dist}\big({\rm spt } \|\Sigma\|, \partial \Omega\big) \ge \min\left\{\frac{\Lambda_\ell - \|\mathbf{H}\|_\infty}{c}, \dist({\rm spt } \|\partial \Sigma\|, \partial \Omega)\right\}.
\end{equation}
\item[$(\mathbb{C})$] If $\|\mathbf{H}\|_\infty = \Lambda_\ell$, $\Lambda^2_\ell \ge c$ and \eqref{condi_volume} is replaced by
\begin{equation}\label{condi_volume_para}
\text{$\Sigma$ is rectifiable and } \qquad \int^{\infty} \frac{r \di r}{\|\Sigma\|(B_r)} = \infty, 
\end{equation}
then 
\begin{equation}\label{bonita_2}
{\rm dist}\big({\rm spt } \|\Sigma\|, \partial \Omega\big) = \dist \big({\rm spt } \|\partial \Sigma\|, \partial \Omega \big).
\end{equation}
Moreover, if $\Sigma$ has no boundary, $\Sigma$ must be contained into an equidistant hypersurface of $\partial \Omega$.
\end{itemize}
\end{theorem}

\begin{remark}[\textbf{Regularity}]
\emph{For the meaning of \eqref{finiteform_vari} in the barrier sense, see Definition \ref{def_barrier} below. Note that, in particular, $\partial \Omega$ is not required to possess a regular neighbourhood of uniform size where the normal exponential map is a diffeomorphism. Also, our assumptions on $M, \partial \Omega$ allow for possible generalization to metric spaces with suitable weak notions of $(\ell-1)$-th Ricci curvature, for instance the one recently considered in \cite{kettemondino} via optimal transport. 
}
\end{remark}

\begin{remark}[\textbf{On condition $\PP_{\ell-1}^-[\II_{\partial \Omega}] \ge \Lambda_{\ell-1}$}] 
\emph{Although it does not appear in any of $(\mathbb{A}),(\mathbb{B}),(\mathbb{C})$, this technical requirement plays an important role in the proof. Perhaps the condition is removable, and the issue seems easier to prove in the hypersurface case $\ell = m-1$.
}
\end{remark}
\begin{remark}[\textbf{On the locally bounded bending condition}]
\emph{The condition is defined in Section \ref{sec_approx}: loosely speaking, it requires the existence of supporting hypersurfaces for $\partial \Omega$ satisfying \eqref{finiteform_vari} and whose second fundamental form is uniformly bounded on compact subsets of $\partial \Omega$. For instance, a $C^{1,1}_\loc$ boundary has locally bounded bending from outwards, but the condition is more general and includes, for instance, the case of convex cones and convex envelopes on Cartan-Hadamard manifolds with bounded sectional curvature. The condition is needed to apply the smooth approximation results that are currently known in the literature, and it would be interesting to remove it. 
}
\end{remark}

We briefly comment on conditions $(\mathbb{A}),(\mathbb{B}),(\mathbb{C})$ in Theorem \ref{omori-yau_vari} by means of some simple examples, leaving further ones to Remark \ref{rem_ABD}. Case $(\mathbb{A})$ applies, for instance, to submanifolds $\Sigma \to \HH^{m}$ of hyperbolic space with mean curvature $\|{\bf H}\|_\infty < 1$, satisfying \eqref{condi_volume} and lying in a horospherically convex domain $\Omega$, that is, in the intersection of (mean convex) horoballs. Indeed, in this case $\Omega$ satisfies \eqref{finiteform_vari} with $\Lambda_\ell = \Lambda_{\ell-1} = 1$, and has locally bounded bending from outwards. Similarly, $(\mathbb{B})$ applies when $\Omega$ can be written as the intersection of mean convex domains bounded by hyperspheres with fixed mean curvature $\Lambda_\ell \in (0,1)$.\par
On the other hand, $(\mathbb{C})$ can be applied when $\Sigma \ra \R^3$ is a complete, immersed minimal surface with compact boundary and finite total curvature, lying in a smooth, $2$-convex domain $\Omega$; this enables us to recover \cite[Thm. 4.1]{adfl}. Indeed, the finite total curvature assumption together with the compactness of $\partial \Sigma$ guarantee, by standard results, that $\Sigma$ has quadratic area growth, namely that $\|\Sigma\|(B_r) \le Cr^2$ for some $C>0$. On the other hand, the maximum principle at infinity in \cite{meeksrosenberg} for pairs of properly immersed minimal surfaces requires further tools, so it cannot be directly obtained from our main result. Observe that the volume growth condition \eqref{condi_volume_para} cannot be weakened to \eqref{condi_volume}, as shown by the example of a higher dimensional catenoid $\Sigma^3 \to \R^4$, that is contained in a slab of $\R^4$ and satisfies $\|\Sigma\|(B_r) \asymp r^3$.

\subsection{The smooth case} \label{rem_smooth}
The strategy of the proof is in principle quite simple and proceeds by contradiction. Suppose that $\Sigma$ is smooth. If the conclusions of Theorem \ref{omori-yau_vari} are not satisfied, we shall find a suitable function $u$ on $\Sigma$, related to the distance $r$ from $\partial \Omega$, that is bounded from above and solves either of the following inequalities on some upper level set $\Omega_\gamma = \{u>\gamma\}$ not intersecting $\partial \Sigma$:
	\begin{equation}\label{le_equaz}
	(\alpha) \ \ \ \ \Delta_\Sigma u \ge \delta \qquad \text{or} \qquad (\beta) \ \ \ \  \Delta_\Sigma u \ge 0,
	\end{equation}
for some constant $\delta>0$, respectively under conditions $\|{\bf H}\|_\infty < \Lambda_\ell$ or $\|{\bf H}\|_\infty = \Lambda_\ell$. Since $\Sigma$ is typically noncompact, to grasp the desired contradiction we then need a Liouville theorem for $u$, that follows from maximum principles at infinity in the spirit of Omori-Yau's ones (cf. \cite{PRS1}). The principles that we need tie to properties coming from stochastic geometry, precisely we shall require that $\Sigma$ be either stochastically complete (case $(\alpha)$) or parabolic (case $(\beta)$). Recall that a boundaryless manifold $\Sigma$ is stochastically complete (respectively, parabolic) if the minimal Brownian motion on $\Sigma$ is non-explosive (respectively, recurrent). A detailed account can be found in \cite{grigoryan,PRS1,amr}, in particular, we underline that $\Sigma$ is stochastically complete under the validity of mild geometric conditions including, for instance, the growth requirement \eqref{condi_volume}. The link to maximum principles has been established in \cite{PRS1,amir}, see also \cite{maripessoa}: the stochastic completeness of $\Sigma$ turns out to be equivalent to the following form of the \emph{weak maximum principle at infinity}: 	
\begin{itemize}
	\item[] for every $u \in C^2(\Sigma)$ that is bounded from above and solves $\Delta_\Sigma u \ge f(u)$ on $\Sigma$, for some $f \in C(\R)$, either
	\[
	\sup_\Sigma u = \sup_{\partial \Sigma} u \qquad \text{or} \qquad f(\sup_\Sigma u ) \le 0.
	\]
\end{itemize}
Clearly, this last property prevents the existence of $u$ bounded from above and satisfying the first in \eqref{le_equaz} on some upper level set. In a similar way, \eqref{condi_volume_para} guarantees the parabolicity of $\Sigma$ (see \cite{PRS1}), that turns out to be equivalent to the following Liouville theorem dating back to L. Ahlfors (cf. \cite[Thm. 6C]{ahlforssario} and \cite{imperapigolasetti}):
\begin{itemize}
	\item[] for every $u \in C^2(\Sigma)$ that is bounded from above and solves $\Delta_\Sigma u \ge 0$ on $\Sigma$, it holds $\sup_\Sigma u = \sup_{\partial \Sigma} u$.
\end{itemize}
%

\noindent Assume that $\Sigma$ has \emph{compact} boundary. In this case, by definition, $\Sigma$ is stochastically complete, respectively parabolic, provided that some (equivalently, every) double $\mathscr{D}(\Sigma)$ of $\Sigma$ is so\footnote{Given $\Sigma$ with compact boundary, recall that a double $\mathscr{D}(\Sigma)$ of $\Sigma$ is any manifold constructed by gluing two copies of $\Sigma$ along their boundary, and keeping the original metric outside of a relatively compact neighbourhood of $\partial \Sigma$. The property that $\mathscr{D}(\Sigma)$ is stochastically complete, or parabolic, does not depend on the choices made in the gluing region, cf. \cite[Sec. 7.3]{PRS2}.}. In summary, if $\Sigma$ is smooth and $\partial \Sigma$ is compact, then the mass growth condition \eqref{condi_volume} can be replaced by either of the following assumptions:
$$
\begin{array}{rl}
(i) & \qquad \text{some (equivalently, every) double of $\Sigma$ is stochastically complete;} \\[0.2cm]
(ii) & \qquad \text{$\Sigma$ is properly immersed into $\Omega$.}
\end{array}
$$
Similarly, if $\Sigma$ is smooth and $\partial \Sigma$ is compact, \eqref{condi_volume_para} can be replaced by the requirement that 
$$
\begin{array}{rl}
(iii) & \qquad \text{some (equivalently, every) double of $\Sigma$ is parabolic.}
\end{array}
$$
It remains to comment on condition $(ii)$, and show that it implies $(i)$. To see this, combining $(ii)$ with $\|{\bf H}\|_\infty < \Lambda_\ell$, and using \cite[Example 1.14]{PRS1}, one can construct on $\Sigma$ a proper function $w$ satisfying
	\begin{equation}\label{strongKhasm}
	\left\{\begin{array}{l}
	w(x) \ra +\infty \qquad \text{as } \, x \ra \infty, \\[0.2cm] 
	|\nabla w| \le C, \qquad \Delta w \le C \qquad \text{on } \, \Sigma,
	\end{array}\right.
	\end{equation}
for some constant $C>0$. Here, the first line means that $w$ has compact sublevel sets in $\Sigma$ (that is, including the boundary). Given a double $\mathscr{D}(\Sigma)$, by doubling $w$ and mollifying it in the gluing region one can therefore match all of the conditions in \eqref{strongKhasm} on $\mathscr{D}(\Sigma)$, up to enlarging $C$. Hence, because of \cite[Thm. 1.9]{PRS1}, the full Omori-Yau maximum principle holds on $\mathscr{D}(\Sigma)$, which implies the stochastic completeness of $\mathscr{D}(\Sigma)$. For a detailed analysis of the relations between the two principles, we refer the reader to \cite{PRS1,amr,maripessoa}.

\begin{remark}\label{rem_ABD}
\emph{A first non-enclosure result in the spirit of Theorem \ref{omori-yau_vari} was given by L. Al\'\i as, G.P. Bessa and M. Dajczer \cite{Bessa}, who proved the following: if $\Sigma$ is a $\ell$-dimensional submanifold properly immersed into a cylinder $\Omega = \mathbb{R}^s \times B_r^{m-s}$, where $\ell > s$ and $B_r$ is a regular, convex geodesic ball into a $(m-s)$-dimensional Riemannian manifold with sectional curvature bounded from above by $-c \in \R$ (with $r < \frac{\pi}{2\sqrt{-c}}$ if $c<0$), then
	\[
	\|{\bf H}\|_\infty \ge \frac{\ell-s}{\ell} \frac{\cn_c(r)}{\sn_c(r)},
	\]
where $\sn_c,\cn_c$ are the Jacobi functions associated to the space form of sectional curvature $-c$, see Section \ref{sec_prelim} below. In particular, $\Omega$ does not contain stochastically complete, minimal $\ell$-submanifolds without boundary. In an unpublished paper, using a version of the sliding method J. Espinar and H. Rosenberg \cite{Espinar} gave a different proof of this result by showing that the distance between $\Sigma$ and the barrier remains positive. We here derive the result as a direct application of Theorem \ref{omori-yau_vari}, taking into account the discussion in Subsection \ref{rem_smooth}: it is enough to observe that, by the Hessian comparison theorem, the second fundamental form of $\partial \Omega$ in the inward direction has the zero eigenvalue with multiplicity $s$, while $m-s$ eigenvalues are at least $\cn_c(r)/\sn_c(r) >0$, so	
	\[
	\PP_\ell^-(\II_{\partial \Omega}) \ge \frac{\ell-s}{\ell} \frac{\cn_c(r)}{\sn_c(r)}.
	\]
Similarly, we recover the mean curvature estimates in \cite{BLPS} for immersions into horocylinders.
}
\end{remark}

If $\Sigma$ is smooth and without boundary, and $\Sigma$ approaches $\partial \Omega$, we can guarantee both that $\Sigma$ is stochastically incomplete, and that the Laplace-Beltrami operator of $\Sigma$ has discrete spectrum. The corollary below generalizes a result due to G.P. Bessa, L.P. Jorge and J.F. Montenegro in \cite{bjm}, who considered the spectrum of hypersurfaces properly contained in regular balls, especially, of minimal surfaces properly contained in a ball of $\R^3$.

\begin{corollary}\label{cor_spectrum}
In the assumptions of Theorem \ref{omori-yau_vari} on $M$ and $\Omega$, let $\Sigma \ra \Omega$ be a smooth $\ell$-dimensional immersed manifold without boundary. If 
\begin{equation}\label{eq29}
\begin{array}{l}
\mathrm{dist}(x, \partial \Omega) \ra 0 \qquad \text{as $x \in \Sigma$, $x \to \infty$,} \\[0.3cm]
\disp \|{\bf H}\|_\infty < \Lambda_\ell,
\end{array}
\end{equation}
where the first condition means that $\{ x \in \Sigma, \mathrm{dist}(x, \partial \Omega) > \eps\}$ is compact for each $\eps > 0$, then $\Sigma$ is stochastically incomplete and the spectrum $\sigma(\Delta_\Sigma)$ of its Laplace-Beltrami operator is discrete.
\end{corollary}

\vspace{0.5cm}

\noindent {\bf Acknowledgements.} We thank the referee for his/her very careful reading of the manuscript, in particular, for suggesting to us the argument at the end of the proof of Theorem \ref{MPV_para}. E.S.G. thanks Prof. Diego Moreira for valuable conversations about Lipschitz functions. This work was done when E.S.G. was temporary professor at Universidade Federal do Cear\'a, Fortaleza. He thanks the institution for the fruitful working environment.

\section{Preliminaries}\label{sec_prelim}

The lower bound on $\Ric^{(\ell-1)}$ in the statement of Theorem \ref{omori-yau_vari} enables to apply comparison results for the distance from $\partial \Omega$. We follow the approach to comparison theory via Riccati equations, that has been extensively developed by J.-H. Eschenburgh and E. Heinze \cite{EschHein,Esch_1,Esch_2}, see also \cite{petersen} and \cite[Ch. 2]{PRS2} for a detailed account. Our version of the comparison theorem below is slightly stronger than those that we found in the above references, which suggested us to provide a concise yet complete proof. A corresponding statement considering the distance to a fixed point can be found in \cite[Prop. 7.4]{maripessoa}. We first recall that, given an open subset $\Omega \subset M$ and a point $y \in \partial \Omega$, a smooth hypersurface without boundary $S$ is said to be \emph{supporting (for $\Omega$) at $y$} if $y \in S$ and $\Omega \cap S = \varnothing$. We also say that $S$ touches $\Omega$ at $y$ from the outside. By modifying $S$ in a small neighbourhood around $y$, we can assume that $S$ is the boundary of a small, connected open set $B_S$ disjoint from $\Omega$ and with $\overline{B}_S$ diffeomorphic to a ball. \\[0.2cm]
\noindent \textbf{Agreement. } Hereafter, a supporting hypersurface $S$ will always be the boundary of $B_S$   as above. In particular, $M\backslash S$ has two connected components (we always assume $M$ to be connected).
%

\begin{definition}\label{def_barrier}
Let $\Omega \subset M^m$ be an open subset with $\partial \Omega \neq \varnothing$, and let $\ell\in \{1, \ldots, m-1\}$.
\begin{itemize}
\item[(i)] Given $\Lambda_\ell \in C(M)$, we say that the second fundamental form $\II_{\partial \Omega}$ in the inward direction satisfies
	\[	
	\PP_{\ell}^-[\II_{\partial \Omega}] \ge \Lambda_{\ell} \qquad  \text{on } \, \partial \Omega
	\]
in the barrier sense if the following holds: for every $y \in \partial \Omega$ and every $\eps>0$, there exists a supporting hypersurface $S_\eps$ at $y$ such that $\PP_\ell^-[ \II_{S_\eps}](y) > \Lambda_\ell(y) - \eps$, where $\II_{S_\eps}$ is the second fundamental form of $S_\eps$ in the direction pointing towards $\Omega$ (i.e., the exterior normal to $B_{S_\eps}$).
\end{itemize}
\item[(ii)] Given $\Lambda_\ell,\Lambda_{\ell-1} \in C(M)$, we say that $\II_{\partial \Omega}$ in the inward direction satisfies
	\[
	\PP_{\ell-1}^-[\II_{\partial \Omega}] \ge \Lambda_{\ell-1}, \qquad \PP_{\ell}^-[\II_{\partial \Omega}](y) \ge \Lambda_{\ell} \qquad  \text{on } \, \partial \Omega
	\]
in the barrier sense if, at every $y \in \partial \Omega$, there exists a supporting hypersurface $S_\eps$ satisfying \emph{both} of the inequalities
	\[
	\PP_{\ell-1}^-[\II_{S_\eps}](y) > \Lambda_{\ell-1}(y)-\eps, \qquad \PP_{\ell}^-[\II_{S_\eps}] > \Lambda_{\ell}(y) - \eps.
	\]
\end{definition}

Consider the signed distance function $r$ from $\partial \Omega$, with the agreement that $\Omega = \{r>0\}$. 
Let $x \in \Omega$, and let $y \in \partial \Omega$ be a nearest point to $x$, i.e. it satisfies $r(x) = \mathrm{dist}(x,y)$. Hereafter, a segment will mean a unit speed geodesic that is minimizing for its endpoints. 

\begin{remark}\label{rem_unique}
\emph{Let $x \in \Omega$ and let $y \in \partial \Omega$ be a nearest point to $x$. If there exists a supporting hypersurface $S$ at $y$, then there exists a unique segment $\gamma_y$ from $y$ to $x$ with interior contained in $\Omega$. Indeed, if $\gamma,\sigma$ are segments from $y$ to $x$ whose restriction to $(0,r(x))$ is contained in $\Omega$, they both are segment also from $S$ to $x$. Hence, they issue orthogonally to $S$, and both point in the component of $M \backslash S$ containing $\Omega$. This forces $\gamma'(0) = \sigma'(0)$, and thus $\gamma \equiv \sigma$.
}
\end{remark}

If $S$ is a supporting hypersurface at $y$, then the signed distance $\bar r$ from $S$ satisfies $\bar r \ge r$ on $\Omega$, with equality at $x$. It is known that $\bar r$ is smooth on the open set $\Omega \backslash \cut(S)$ and up to $\partial S$, with $\cut(S)$ the cut-locus of $S$. If $\gamma : [0, r(x)] \ra M$ is a segment from $y$ to $x \in \Omega$, it is known that $\gamma\big([0, r(x))\big) \subset \Omega \backslash \cut(S)$. All of these facts can be easily adapted from the corresponding ones for the distance to a point, and for a detailed treatment we suggest \cite[Chapters 3 and 5]{petersen}. We use the index agreement 
	\[
	1 \le i,j,k,t \le m, \qquad 2 \le \alpha,\beta \le m. 
	\]
Differentiating twice the identity $|\nabla \bar r|^2 =1$ in a neighbourhood of $\gamma\big([0, r(x))\big)$, by Ricci commutation laws we deduce
\begin{equation}\label{basehessian}
\bar r_i \bar r_{ij} = 0, \qquad \bar r_i \bar r_{ijk} + \bar r_{ij} \bar r_{ik} + \bar r_i \bar r_tR_{ijtk} = 0.
\end{equation}
Choose an orthonormal basis $\{E_j(t)\}$ which is parallel along $\gamma$ and such that $E_1 = \gamma'$, and let $B(t) \in \mathrm{Sym}(\R^{m-1})$ represent the $(1,1)$-version of $\nabla^2 \bar r(\gamma(t))$ restricted to $E_1^\perp$ in the basis $\{E_\alpha(t)\}$. Then, by \eqref{basehessian},
\begin{equation}\label{eq_riccati}
\left\{ \begin{array}{l}
B' + B^2 + R_\gamma = 0 \qquad \text{on $(0,r(x))$,} \\[0.2cm]
B(0)_{\alpha\beta} = -\II_{S}\big(E_\alpha(0),E_\beta(0)\big).
\end{array}\right.
\end{equation}
where $(R_\gamma)_{\alpha\beta}(t) = R\big(E_\alpha(t), \gamma'(t), E_\beta(t), \gamma'(t)\big)$. 

The next lemma is probably well-known, but we didn't find a precise reference and so we provide a full proof.

\begin{lemma}\label{lem_noncut}
Fix $x \in \Omega$ and a nearest point $y \in \partial \Omega$ to $x$. For any supporting hypersurface $S$ at $y$, then there exists $S'$, close to $S$ in the $C^\infty$ topology in a neighbourhood of $y$, still supporting at $y$, and such that $x \not \in \cut(S')$.
\end{lemma}

\begin{proof}
If $x \not \in \cut(S)$, clearly take $S' = S$. Otherwise, it is known that either there exist at least two, and finitely many, distinct segments from $S$ to $x$, or $x$ is focal for $S$. For $\delta$ small enough, consider a regular geodesic ball $B_{2\delta}(y) \subset S$ and a Fermi chart
	\[
\Psi : (-2\delta,2\delta) \times B_{2\delta}(y) \to U \subset M, \qquad (s,\bar y) \mapsto \exp_{\bar y} \big(s\eta(\bar y)\big),
	\]
where $\eta$ is the unit normal pointing outwards from the open set $B_S$ in the definition of supporting hypersurface. Let $U_s = \Psi\big( (-s,s) \times B_s(y)\big)$, $U^+ = U \cap \{s > 0\}$. Fix $0<\eps << \delta$ and choose a compact, boundaryless hypersurface $S' \subset \overline{B}_S$ in such a way that 
	\[
	S' \cap U_\delta = U_\delta \cap \left\{ \ s = - \frac{\eps}{2} \mathrm{dist}_S(y,\bar y)^2\right\}, \qquad \overline{B}_{S'} \backslash \{y\} \subset B_S. 
	\] 
Then, the second condition guarantees that the distance $\bar r'$ from $S'$ satisfies $\bar r' \ge \bar r$ on $\Omega$, with equality only at $x$ (we call hereafter such a construction an $\eps$-bending of $S$ from outwards).
Observe that the second fundamental form of $S'$ at $y$ in the direction pointing towards $\Omega$ is $\II_{S} - \eps g$, where $g$ is the metric of $M$. Let $\gamma$ be a segment from $y \in S$ to $x$. From $\overline{B}_{S'} \backslash \{y\} \subset B_S$ and Remark \ref{rem_unique}, we infer that $\gamma$ is the unique segment from $S'$ to $x$. Thus, either $\bar r'$ is smooth around $x$, or $x$ is focal for $S'$. We next prove that $x$ cannot be focal. Let $B,B_\eps$ be the matrix representation of, respectively, $\nabla^2 \bar r$ and $\nabla^2 \bar r'$ on $\gamma'^\perp$, in the parallel orthonormal frame $\{E_\alpha\}$, and observe that $B_\eps(0) = -\II_S + \eps I$. 
Denote with $J,J_\eps$ the Jacobi tensors on $\gamma'^ \perp$ corresponding to $B, B_\eps$:
$$
\left\{ \begin{array}{l}
J' = BJ \qquad \text{on } \, [0,r(x)] \\[0.1cm]
J(0)=I,
\end{array}\right. \qquad \left\{ \begin{array}{l}
J_\eps' = B_\eps J_\eps \qquad \text{on } \, [0,r(x)] \\[0.1cm]
J_\eps(0)=I,
\end{array}\right.
$$
that respectively solve 
$$
\left\{ \begin{array}{l}
J'' + R_\gamma J = 0 \qquad \text{on } \, [0,r(x)] \\[0.1cm]
J(0)=I, \ \ J'(0) = B(0), 
\end{array}\right. \qquad \left\{ \begin{array}{l}
J_\eps'' + R_\gamma J_\eps = 0 \qquad \text{on } \, [0,r(x)] \\[0.1cm]
J_\eps(0)=I, \ \ J_\eps'(0) = B_\eps(0). 
\end{array}\right.
$$
Let $v \in \R^{m-1}$, $|v|=1$ be such that the Jacobi field $V_\eps(t) = J_\eps(t)v$ is zero at $r(x)$. We consider the modified field $\bar V_\eps(t) = V_\eps(t)e^{-\eps t}$. Then, 
\begin{equation}\label{eq_barVeps}
\begin{array}{l}
\disp \bar V_\eps(r(x))=0, \qquad \bar V_\eps(0) = v, \qquad \bar V_\eps'(0) = (B_\eps(0) - \eps I)\bar V_\eps(0) = B(0)v, \\[0.2cm]
\disp \bar V_\eps'' + R_\gamma \bar V_\eps = -2\eps \bar V_\eps' - \eps^2\bar V_\eps.
\end{array}
\end{equation}  
In particular, the conditions in the first line imply that $\bar V_\eps$ can be viewed as the variational vector field of a variation $f(s,t) : (-s_0,s_0) \times [0,r(x)] \ra M$ such that 
$$
f(s,r(x)) = x \quad \text{for each } \, s, \qquad f(s,0) \subset S.
$$
Computing the second variation of the energy $E(f_s) = \frac{1}{2} \int_0^{r(x)}|f'_s|^2$, we get
$$
\partial^2_s E(f_s)(0) = \langle B(0)v,v \rangle + \int_0^{r(x)} \Big[ |\bar V_\eps'|^2 - \langle R_\gamma \bar V_\eps, \bar V_\eps \rangle \Big] = - \int_0^{r(x)} \langle \bar V_\eps'' + R_\gamma \bar V_\eps, \bar V_\eps \rangle. 
$$
Using \eqref{eq_barVeps} and integrating by parts, we obtain
\begin{equation}\label{eq_impoperdopo}
\partial^2_s E(f_s)(0) = -\eps |\bar V_\eps(0)|^2 + \eps^2 \int_0^{r(x)}|\bar V_\eps|^2 = -\eps  + \eps^2 \int_0^{r(x)}|\bar V_\eps|^2.
\end{equation}
Next, by Rauch comparison and $R_\gamma \ge \bar c I$, with $\bar c = \inf_{[0,t]} R_\gamma$, we deduce that $|\bar V_\eps|$ is uniformly bounded for $\eps \in (0, \eps_0)$. Thus, for $\eps$ small enough, the left-hand side is negative, hence $\gamma$ cannot be a segment from $x$ to $S$, contradiction.
\end{proof}

Let us first fix some notation: for $c \in \R$, we denote with $\sn_c(t)$ the solution of 
\begin{equation}\label{solitefunzioni}
\left\{ \begin{array}{l}
\sn_c''(t) - c \, \sn_c(t) = 0 \qquad \text{on } \, \R,\\[0.1cm]
\sn_c(0)=0, \quad \sn_c'(0)=1, 
\end{array}\right. \qquad \text{and set } \ \ \  \cn_c(t) := \sn_c'(t). 
\end{equation}
Note that
$$
\sn_c(t) = \left\{ \begin{array}{ll}
\sin(t\sqrt{-c})/\sqrt{-c} & \quad \text{if } \, c<0, \\[0.2cm]
t & \quad \text{if } \, c=0 \\[0.2cm]
\sinh(t\sqrt{c})/\sqrt{c} & \quad \text{if } \, c>0.
\end{array}\right.
$$
Given a $(2,0)$-tensor $T$ at a point $x \in M$, and given a $\ell$-dimensional subspace $\mcal W \le T_xM$, with  $\Tr_{\mcal W} T$ we indicate the trace of $T$ restricted to $\mcal W$:
$$
\Tr_{\mcal W} T = \sum_{i=1}^\ell T(e_i,e_i), \qquad \text{with $\{e_i\}$ an orthonormal basis of $\mcal W$.}
$$

The comparison theorem that we will need throughout the paper is the following:

\begin{proposition}\label{prop_comparison}
Let $(M^m,\metric)$ be a complete manifold such that, for some $\ell \in \{1,\ldots, m-1\}$ and some $c \in \R$, 
\begin{equation}\label{riccik}
\Ric^{(\ell)} \ge - c. 
\end{equation}
Let $\Omega \subset M$ be an open set with non-empty boundary $\partial \Omega$, whose second fundamental form in the inward direction satisfies 
	\[
	\PP_\ell^-[\II_{\partial \Omega}](y) \ge \Lambda_\ell(y) \qquad \forall \, y \in \partial \Omega
	\]
in the barrier sense, for some $\Lambda_\ell \in C(M)$. Let $r$ be the signed distance from $\partial \Omega$, with the agreement that $r>0$ on $\Omega$. Then, setting $\tau_y = - \Lambda_\ell(y)$, for every $x \in \Omega$ and $y \in \partial \Omega$ nearest point to $x$ it holds
\begin{equation}\label{ine_importante}
\sup_{\footnotesize{\begin{array}{c}
\mcal W \le \gamma_y'(r(x))^\perp \\
\text{$\mathcal{W}$ $\ell$-dimensional}
\end{array}}
} 
\left( \frac 1 \ell \Tr_{\mcal W} \nabla^2 r \right) \le \frac{\tau_y \cn_c\big(r(x)\big) + c \, \sn_c \big(r(x)\big)}{\cn_c \big(r(x)\big) + \tau_y \sn_c \big(r(x)\big)}
\end{equation}
in the barrier sense, that is: for every $\eps>0$ and every nearest point $y$ to $x$, there exist a supporting hypersurface $S^y_\eps$ at $y$ whose associated signed distance $\bar r^y_\eps = \mathrm{dist}(S_\eps^y, \cdot)$ satisfies
	\[
	\sup_{\footnotesize{\begin{array}{c}
\mcal W \le \gamma_y'(r(x))^\perp \\
\text{$\mathcal{W}$ $\ell$-dimensional}
\end{array}} 
}
\left( \frac 1 \ell \Tr_{\mcal W} \nabla^2 \bar r^y_\eps \right) \le \frac{\tau_y^\eps \cn_c\big(r(x)\big) + c \, \sn_c \big(r(x)\big)}{\cn_c \big(r(x)\big) + \tau_y^\eps \sn_c \big(r(x)\big)},
	\]
with $\tau_y^\eps := \tau_y + \eps$. 
\end{proposition}

\begin{proof}
Fix $\eps>0$, let $x \in \Omega$ and let $y \in \partial \Omega$ be a nearest point to $x$. In view of our assumptions and of Lemma \ref{lem_noncut}, we can choose a supporting hypersurface $S_\eps$ at $y$ such that 
	\[
	x \not \in \cut(S_\eps), \qquad \PP_\ell^-[\II_{S_\eps}](y) > \Lambda_\ell(y) - \eps = -\tau_y^\eps.
	\] 
Let $\gamma : [0, r(x)] \ra M$ be the unique segment from $S$ to $x$, and consider the Riccati equation \eqref{eq_riccati} satisfied by $B(t)$. Fix an orthonormal basis $\{v_j\}_{j=1}^\ell$ for an $\ell$-subspace $\mcal{W} \le \gamma'(r(x))^\perp$, and extend them to parallel fields $\{V_j\}_{j=1}^\ell \subset \gamma'^\perp$ along $\gamma$. We still call $\mathcal{W}$ the span of $V_1,\ldots, V_\ell$, and we denote with $\pi_\ell : T_{\gamma(t)}M \ra \mathcal{W}$ the orthogonal projection. From \eqref{riccik} and the definition of $\Ric^{(\ell)}$,
$$
\sum_{j=1}^\ell \langle R_\gamma V_j, V_j \rangle \ge \ell \Ric^{(\ell)}(\gamma') \ge - \ell c,
$$
Tracing \eqref{eq_riccati} on $\{V_j\}$ and using the last inequality we deduce that the function 
$$
\theta_y(t) : = \frac{1}{\ell}\sum_{j=1}^\ell \langle B V_j, V_j \rangle(t)
$$
satisfies
$$
\left\{ \begin{array}{l}
\disp \theta_y'(t) + \ell^{-1}\sum_{j=1}^\ell \langle B^2V_j,V_j \rangle - c \le 0 \qquad \text{on } \, (0, r(x)], \\[0.2cm]
\theta(0) = \ell^{-1}\sum_{j=1}^\ell \langle BV_j,V_j\rangle(0).
\end{array}\right.
$$
By the min-max characterization of eigenvalues and since $B(0) = -\II_{S_\eps}$, 
$$
\frac 1 \ell \sum_{j=1}^\ell \langle BV_j,V_j\rangle(0) \le - \PP_\ell^-[\II_{S_\eps}](y) \le \tau^\eps_y.
$$
Furthermore, using Newton's inequality, we get
$$
\sum_{j=1}^\ell \langle B^2V_j,V_j \rangle \ge \sum_{j=1}^\ell \langle (\pi_\ell \circ B)^2 V_j,V_j \rangle \ge \ell \theta_y^2,
$$
hence 
\begin{equation}\label{inericcati}
\left\{ \begin{array}{l}
\disp \theta_y' + \theta_y^2 - c \le 0 \qquad \text{on } \, (0, r(x)], \\[0.2cm]
\disp \theta_y(0) \le \tau^\eps_y.
\end{array}\right.
\end{equation}
By Riccati's comparison for ODE (see for instance Lemma 2.1 and Corollary 2.2 in \cite{PRS2}), 
$$
\theta_y(t) \le \frac{\tau_y^\eps  \cn_c(t) + c  \, \sn_c(t)}{\cn_c(t) + \tau_y^\eps \sn_c(t)} \qquad \text{on } \, (0, r(x)],
$$
the right-hand side being the solution of \eqref{inericcati} with equality signs. In particular, the inequality implies that the denominator of the right-hand side never vanishes on $(0,r(x)]$. The desired inequality \eqref{ine_importante} follows by setting $t = r(x)$ from the arbitrariness of $\eps$, $\mcal W$ and of $y$.
\end{proof}

\section{Construction of the barrier}

Following most of the literature quoted in the Introduction, our argument depends on the construction of a geometrically useful function $u$ satisfying the inequality 
	\[
	\PP_\ell^-[\nabla^2 u] - h|\nabla u|  \ge \delta 
	\]
on a suitable subset of $\Omega$, for some constant $\delta \ge 0$. The existence of such $u$ is often guaranteed by the definition of $\Omega$ itself, as for instance in \cite{DS}, where $\Omega = u^{-1}((-\infty, 0))$ and $u$ is as in \eqref{def_Phi_DS}: under the requirements $s < \ell$, $c \le \frac{\ell-s}{s}$, it can be checked both that $\partial \Omega$ is $\ell$-mean convex, and that
	\begin{equation}\label{eq_Pkmeno_HL}
	\PP_\ell^-[\nabla^2 u] \ge 0 \qquad \text{on } \, \Omega,
	\end{equation}
the latter implying the subharmonicity of $u$ when restricted to minimal $\ell$-dimensional submanifolds. In the generality of Theorem \ref{omori-yau_vari}, however, it seems not obvious that such $u$ must exist. Hence, the goal of the present section is to provide a positive answer and a general construction that works under fairly weak assumptions on $M$ and $\Omega$. Similar arguments were used by L. Mazet in \cite[Sec. 6.1]{mazet}. Note that we cannot exploit (at least, not directly) the theory developed for \eqref{eq_Pkmeno_HL} by R. Harvey and B. Lawson in a series of papers (cf. \cite{HL_dir,HL_plurisub}), since their existence results using Perron's method are based on the knowledge, a priori, of a subsolution of \eqref{eq_Pkmeno_HL}.\par
Recall that a function $u$ is said to solve 
	\begin{equation}\label{eq_Pkm_general}
	\PP_\ell^-[\nabla^2 u] \ge f(x, u, \nabla u) \qquad \text{in the barrier sense on } \, \Omega
	\end{equation}
if, for every $x \in \Omega$ and every $\eps>0$, there exists a smooth $u_\eps$ in a neighbourhood of $x$ such that $u_\eps \le u$, $u_\eps(x) = u(x)$ and  	
	\[
	\PP_\ell^-[\nabla^2 u_\eps](x) \ge f(x, u_\eps(x), \nabla u_\eps(x)) - \eps.
	\]
Evidently, if $u$ solves \eqref{eq_Pkm_general} in the barrier sense, it also solve the inequality in the viscosity sense. 

\begin{proposition}\label{prop_constr_u}
Let $M^m$ be a complete manifold satisfying 
\begin{equation}\label{riccik_gen}
\Ric^{(\ell-1)} \ge - c, 
\end{equation}
for some $\ell \in \{2,\ldots, m-1\}$ and some $c \in \R$. Let $\Omega \subset M$ be an open set whose second  fundamental form $\II_{\partial \Omega}$ in the inward direction satisfies  
\begin{equation}\label{secondfund_gen}
\PP_{\ell-1}^-[\II_{\partial \Omega}] \ge \Lambda_{\ell-1}, \qquad \PP_\ell^-[\II_{\partial \Omega}] \ge \Lambda_\ell \ge 0
\end{equation}
in the barrier sense, for some $\Lambda_{\ell-1} \in \R$, $\Lambda_\ell\in [0,\infty)$. Choose a constant $h$ satisfying 
$$
0 \le h \le \Lambda_\ell, \qquad \text{with } \ \ h < \Lambda_\ell \ \  \text{ if } \  \ \Lambda_\ell^2 < c,
$$
and let $R \in \R$ satisfying 
$$
\begin{array}{ll}
R> 0 & \qquad \disp \text{if} \qquad \Lambda_\ell^2 \ge c \\[0.3cm]
\disp R \in \left( 0, \frac{\Lambda_\ell - h}{c} \right) & \qquad \disp \text{if} \qquad \Lambda_\ell^2 < c.
\end{array}
$$
Then, on the set 
$$
\Omega_R : = \Big\{ x \in \Omega \ : \ \mathrm{dist}(x, \partial \Omega) < R \Big\} 
$$
there exists $0 < u \in \lip_\loc(\Omega)$ such that 
\begin{equation}\label{proprie_u}
\begin{array}{l}
\text{$u$ only depends on the distance $r$ to $\partial \Omega$, and is strictly decreasing in $r$;} \\[0.2cm] 
|\nabla u| = C_2 u \quad \text{where $u$ is differentiable, for constant $C_2 (c,\ell, \Lambda_{\ell-1},h, \Lambda_\ell - h) > 0$}\\[0.2cm]
\end{array}
\end{equation}
and $u$ satisfies the following inequality in the barrier sense on $\Omega_R$: 
\begin{equation}\label{proprie_u_sol}
\PP_\ell^-[\nabla^2 u] - h|\nabla u| \ge \left\{ \begin{array}{ll}
\bar \delta & \quad \text{if } \, h < \Lambda_\ell, \\[0.2cm]
0 & \quad \text{if } \, h = \Lambda_\ell,
\end{array}\right.
\end{equation}
for some positive constant $\bar \delta(c,\ell, \Lambda_{\ell-1}, h, \Lambda_\ell - h, R)$.
\end{proposition}

\begin{proof}
Let $r(\cdot) = \dist(\cdot, \partial\Omega)$ be the signed distance function from $\partial \Omega$, and define $u = \eta(r)$ for $\eta(t) \in C^2([0,\infty))$ with $\eta' < 0$ to be chosen later. Let $x \in \Omega_R$, let $y \in \partial \Omega$ be a nearest point to $x$, and let $\gamma_y$ be as in Remark \ref{rem_unique}. Let $S_\eps$ be a supporting hypersurface at $y$ with 
	\[
	x \not\in \cut(S_\eps), \qquad \PP_{\ell-1}^-[\II_{S_\eps}](y) > \Lambda_{\ell-1} -\eps, \quad \PP_{\ell}^-[\II_{S_\eps}](y) > \Lambda_{\ell} -\eps.
	\]
and let $r_\eps = \mathrm{dist}(S_\eps, \cdot)$, $u_\eps = \eta(r_\eps)$. Note that $u_\eps$ is smooth near $x$ and touches $u$ from below. Let $\mcal W \subset T_xM$ be a $\ell$-dimensional subspace generated by the first $\ell$ eigenvectors of $  \nabla^2 u_\eps$ counted with multiplicity, and choose an orthonormal basis $\{v_i\}$ for $\mcal W$, that up to rotation we can arrange to satisfy
$$
\begin{array}{c}
v_1 = (\cos \psi)  \nabla r_\eps + (\sin \psi)e_1, \quad v_i = e_i \ \text{ for } \, i \ge 2, \\[0.1cm]
\text{with } \ \{e_1,\ldots, e_\ell\} \in \gamma_y'(x)^\perp \ \text{ an orthonormal set,} 
\end{array}
$$
for some $\psi \in \R$. From $\nabla^2 u_\eps = \eta''(r_\eps) \di r_\eps \otimes \di r_\eps + \eta'(r_\eps)   \nabla^2 r_\eps$ and $\eta' < 0$ we get, at $x$,
$$
\begin{array}{l}
\ell \Big( \PP_\ell^-[  \nabla^2 u_\eps] - h|  \nabla u_\eps| \Big) ={\rm div}^{\mcal W}(\nabla u_\eps)-\ell h| \nabla u_\eps| = \disp \sum_{i = 1}^\ell   \nabla^2 u_\eps(v_i,v_i) - \ell h|  \nabla u_\eps| \\[0.2cm]
\qquad = \disp \eta''(r_\eps) \sum_{i = 1}^\ell \langle   \nabla r_\eps, v_i\rangle^2  + \eta'(r_\eps)\left[\sin^2\psi   \nabla^2 r_\eps(e_1,e_1) +  \sum_{i=2}^\ell   \nabla^2 r_\eps (e_i,e_i)\right] + \ell h\eta'(r_\eps) \\[0.2cm]
\qquad = \disp \eta''(r_\eps)\cos^ 2\psi + \eta'(r_\eps)\left[\sin^2\psi   \nabla^2 r_\eps(e_1,e_1) +  \sum_{i=2}^\ell   \nabla^2 r_\eps (e_i,e_i)\right] + \ell h \eta'(r_\eps) \\[0.3cm]
\qquad = \disp (\cos^2\psi)\left[\eta''(r_\eps) + \eta'(r_\eps)\left( \ell h + \sum_{i=2}^\ell  \nabla^2 r_\eps(e_i,e_i)\right)\right] \\[0.4cm]
\qquad \ \ \disp + (\sin^2\psi)\eta'(r_\eps)\left(\ell h + \sum_{i=1}^\ell   \nabla^2 r_\eps (e_i,e_i)\right). \\[0.2cm]
\end{array}
$$
By \eqref{riccik_gen} and the comparison Proposition \ref{prop_comparison} (note that \eqref{riccik_gen} implies $\Ric^{(\ell)} \ge -c$), we obtain
\begin{equation}\label{bonitas}
\begin{array}{rcl}
\disp \frac{1}{\ell-1}\sum_{i=2}^\ell   \nabla^2 r_\eps (e_i,e_i) & \le & \disp \frac{\bar \tau^\eps \cn_c(r) + c \, \sn_c(r)}{\cn_c(r) + \bar \tau^\eps \sn_c(r)} : = \bar \tau^\eps(r), \\[0.3cm]
\disp \frac{1}{\ell}\sum_{i=1}^\ell   \nabla^2 r_\eps (e_i,e_i) & \le & \disp \frac{\tau^\eps \cn_c(r) + c  \, \sn_c(r)}{\cn_c(r) + \tau^\eps \sn_c(r)} : = \tau^\eps(r),
\end{array}
\end{equation}
where 
\begin{equation}\label{def_varioustau}
\bar \tau^\eps = - \Lambda_{\ell-1} + \eps, \qquad \tau^\eps = - \Lambda_\ell + \eps. 
\end{equation}
We first observe that the partial derivatives $f_\tau, f_t$ of the function
$$
f(\tau,t) = \frac{\tau \, \cn_c(t) + c \, \sn_c(t)}{\cn_c(t) + \tau \,\sn_c(t)}
$$
satisfy 
\begin{equation}\label{derivatives}
f_\tau(\tau,t) > 0, \qquad f_t(\tau,t) = \frac{c-\tau^2}{(\cn_c(t) + \tau \, \sn_c(t))^2},
\end{equation}
from which we deduce
$$
\bar \tau^\eps(r) \le \max\Big\{ \bar \tau^\eps, \sqrt{c_+}\Big\} = \max\Big\{ - \Lambda_{\ell-1} + \eps, \sqrt{c_+}\Big\}.
$$
Therefore, by \eqref{bonitas}, there exists a constant $C_1 : = \ell h + (\ell-1)\max\big\{ - \Lambda_{\ell-1} + 1, \sqrt{c_+}\big\}$ such that, for every $\eps \in (0,1)$,
$$
\ell h + \sum_{i=2}^\ell  \nabla^2 r_\eps(e_i,e_i) \le C_1,
$$
and since $r_\eps = r$ at $x$,
$$
\begin{array}{lcl}
\ell \big( \PP_\ell^-[  \nabla^2 u_\eps] - h|\nabla u_\eps|\big) & \ge & \disp (\cos^2\psi)\left[\eta''(r) + C_1\eta'(r)\right] + (\sin^2\psi)\eta'(r)\ell \big(h + \tau^\eps(r)\big).
\end{array}
$$
For $\delta \ge 0$ to be chosen later, set $\eta(t) = \exp\big\{ -(C_1+\delta)t \big\}$. Then, $\eta'<0$ and $\eta''= -(C_1+\delta)\eta'$, therefore 
\begin{equation}\label{eq_gradient}
|  \nabla u_\eps| = (C_1+ \delta)u_\eps : = C_2 u_\eps. 
\end{equation}
If $x$ is a point where $u$ is differentiable, from \eqref{eq_gradient} and since $u_\eps$ touches $u$ from below at $x$ we deduce $|\nabla u(x)| = C_2 u(x)$, showing the validity of \eqref{proprie_u}. Furthermore, 
\begin{equation}\label{bellina}
\ell\big( \PP_\ell^-[ \nabla^2 u_\eps] - h|  \nabla u_\eps|\big) \ge \eta'(r) \left[ -(\cos^2\psi)\delta  + (\sin^2\psi)\ell\left(h + \tau^\eps(r)\right)\right] \qquad \text{at } \, x.
\end{equation}
To estimate the term between square brackets, we restrict to $x \in \Omega_R$ and split into cases. 

\begin{itemize}
\item[-] Case $(\mathbb{A}_1)$: $\Lambda_{\ell}^2 > c$. For $\eps$ small enough, $\tau^\eps = -\Lambda_\ell + \eps$ satisfies $(\tau^\eps)^2 > c$, thus by \eqref{derivatives} we get $(\tau^\eps)'(t) = f_t(\tau^\eps,t) \le 0$. Hence, $\tau^\eps(r) \le - \Lambda_\ell + \eps$. If $h < \Lambda_\ell$, we choose $\delta := \ell\frac{\Lambda_\ell -h}{2}$, $\eps < \frac{\Lambda_\ell -h}{2}$ to deduce from \eqref{bellina} the inequality
\begin{equation}\label{bellina_caseA}
\begin{array}{lcl}
\disp \PP_\ell^-[  \nabla^2 u_\eps] - \ell h|  \nabla u_\eps| & \ge & \disp \ell^{-1} \eta'(r) \left[ -(\cos^2\psi)\delta  + (\sin^2\psi)\ell\left(h - \Lambda_\ell + \eps\right)\right] \\[0.2cm]
 & \ge & \disp \ell^{-1}\eta'(r) \left[ -(\cos^2\psi)\delta  - (\sin^2\psi)\delta \right] \\[0.2cm]
& \ge & \disp \ell^{-1}\Big(-\sup_{[0,R]}\eta'\Big)\delta = \ell^{-1}(C_1 + \delta)e^{-(C_1 + \delta)R}\delta : = \bar \delta > 0.
\end{array}
\end{equation}
as claimed. On the other hand, if $h = \Lambda_\ell$ we choose $\delta = 0$ and obtain 
	\begin{equation}\label{bellina_casezero}
	\PP_\ell^-[  \nabla^2 u_\eps] - h|  \nabla u_\eps| \ge \eta'(r)\eps \ge - C_2\eps,
	\end{equation}
Proving \eqref{proprie_u_sol}.
\item[-] Case $(\mathbb{A}_2)$: $\Lambda_\ell^2 = c$. In this case, $\tau^\eps = - \sqrt{c} + \eps$, and 
	\[
	f_t(\tau^\eps, t) = \frac{2\eps \sqrt{c} - \eps^2}{(\cn_c(t) + \tau^\eps \sn_c(t))^2} \quad \left\{ \begin{array}{ll}
\le 0 & \quad \text{if } \, c = 0, \\[0.4cm]
 	\le \frac{2\eps \sqrt{c}}{(\cn_c(R) - \sqrt{c}\sn_c(R))} < \eps' & \quad \text{if } \, c>0,
	\end{array}\right.
	\]
where $\eps'$ can be chosen as small as we wish, provided that $\eps < \eps_1(c,R)$. Therefore, $\tau^\eps(r) \le -\Lambda_\ell + \eps$ for $c =0$, while $\tau^\eps(r) \le -\Lambda_\ell + \eps + \eps'r$ if $c>0$. In both of the cases, if $h = \Lambda_\ell$, from \eqref{bellina} we readily get
	\[
	\PP_\ell^-[  \nabla^2 u_\eps] - h|  \nabla u_\eps| \ge - C_3\eps - C_4 \eps', 
	\]
for suitable constants $C_3,C_4>0$ independent of $\eps,\eps'$, and the sought is proved. On the other hand, if $h< \Lambda_\ell$, choose again $\delta = \ell\frac{\Lambda_\ell- h}{2}$ and $\eps, \eps'$ small enough to obtain
\begin{equation}\label{bellina_caseA3}
\begin{array}{lcl}
\disp \PP_\ell^-[  \nabla^2 u_\eps] - h|  \nabla u_\eps| & \ge & \disp \ell^{-1}\eta'(r) \left[ -(\cos^2\psi)\delta  + (\sin^2\psi)\ell\left(h -\Lambda_\ell + \eps + \eps'R\right)\right] \\[0.2cm]
& \ge & \ell^{-1}(C_1+\delta)e^{-(C_1+\delta)R} \delta := \bar \delta > 0. 
\end{array}
\end{equation} 
\item[-] Case $(\mathbb{B})$: $\Lambda_\ell^2 < c$. In this case, we first observe that the partial derivative $f_t(\tau,t)$ in \eqref{derivatives} is, for fixed $t \in (0,R)$, strictly increasing when $\tau \in [-\sqrt{c}, \tau^*)$, with $\tau^* = - c \, \sn_c(t)/ \cn_c(t)$, and decreasing when $\tau \in (\tau^*, \infty)$. Hence, $\tau^\eps(t)' \le f_t(\tau^*,t)= c$ and therefore $\tau^\eps(r) \le -\Lambda_\ell + \eps + cR$ on $[0,R]$. Plugging into \eqref{bellina} and choosing $\eps < \frac{\delta}{\ell} : = \frac{\Lambda_\ell - h - cR}{2}$ we deduce
\begin{equation}\label{bellina_caseB}
\begin{array}{lcl}
\disp \PP_\ell^-[  \nabla^2 u_\eps] - h|  \nabla u_\eps| & \ge & \disp \ell^{-1}\eta'(r) \left[ -(\cos^2\psi)\delta  + (\sin^2\psi)\ell\left(h - \Lambda_\ell + \eps + cR\right)\right] \\[0.3cm]
& \ge & \disp \ell^{-1}(C_1+\delta)e^{-(C_1+\delta)R}\delta : = \bar \delta > 0.
\end{array}
\end{equation}
\end{itemize}
This concludes the proof.
\end{proof}

\section{Approximation}\label{sec_approx}

Our next goal is to approximate $u$ in the $C^0$-fine topology by means of smooth solutions of \eqref{proprie_u}, up to reducing $\bar \delta$. To do so, it is useful to observe that $u$ also satisfies the inequality 
\begin{equation}\label{eq_u_enhanced}
\PP_\ell^-[\nabla^2 u] - hC_2 u \ge \bar \delta \ge 0 \qquad \text{in the barrier sense on } \, \Omega_R.
\end{equation}
In particular, the fact that $|\nabla u|$ can be extended continuously to the cut-locus simplifies things considerably. Smooth approximation of functions that satisfy, in a suitable weak sense, some geometrically relevant differential inequalities were thoroughly studied by R. Greene and H. Wu in a series of papers (see in particular \cite{GW}). In \cite{sha,Wu} the authors defined a notion of $\ell$-convexity that is suited for application of Greene-Wu smoothing procedure, and we here adapt their definition to cover our case of interest. 
%
%
%
%

Before we start, we need to introduce some terminology from \cite{GW}. Let $\Omega$ be an open set of $M$. Along of this section $\mathscr{C}$ will denote the sheaf of the germs of continuous functions on $\Omega$, and $\mathscr{S}$ a particular subsheaf of $\mathscr{C}$ that will be specified later. The elements of $\mathscr{S}$ will be denoted by $[f]_p$, where $p\in M$ and $f$ is a continuous function defined in a neighbourhood of $p$ in $M.$ Furthermore, if $\pi|_{\mathscr{S}}:\mathscr{S}\to M$ is the standard projection restricted to $\mathscr{S}$, then we will indicate by $\mathscr{S}_p$ the set $(\pi|_{\mathscr{S}})^{-1}(p).$ Finally, if $V \subset\Omega$ is open, the set of continuous functions $f:V\to\R$ with the property that $[f]_p\in\mathscr{S}_p$ for all $p\in V$ will be indicated by $\Gamma(\mathscr{S},V).$

Fix an open set $\Omega$ in $M$ and a positive function $\beta\in C(\Omega)$. Given $x_0 \in \Omega$, we say that
\begin{equation}\label{eq_fsoluz-Lip}
|\nabla f| < \beta
\end{equation}
in Greene-Wu sense (GW-sense) at $x_0$ if there exists a neighbourhood $V$ of $x_0$ and $\eps>0$ such that
	\[
	\lip(f, V) < \beta(x_0)- \eps, 
	\]
with $\lip(f,V)$ the Lipschitz constant of $f$ on $V$. We say that \eqref{eq_fsoluz-Lip} holds in the GW-sense on $\Omega$ if it holds at every $x_0 \in \Omega$. Also, for $v \in T_xM$, let $\gamma$ be the geodesic issuing from $x$ with velocity $v$, and set
$$
Cf(x,v) = \liminf_{t \ra 0} \frac{f(\gamma(t)) + f(\gamma(-t)) - 2 f(x)}{t^2}, 
$$
Given $1 \le \ell \le m$ and $\kappa \in C(\Omega)$, we say that $f$ solves 
\begin{equation}\label{eq_fsoluz}
\PP_\ell^-[\nabla^2 f] > \kappa
\end{equation}
in GW-sense at $x_0$ if there is a neighbourhood $V$ of $x_0$, and constants $\eps,\eta' >0$, such that for each $x \in V$ and $v_1, \ldots, v_\ell \in T_xM$ satisfying $|\langle v_i, v_j\rangle - \delta_{ij}| < \eps$ it holds   
$$
\frac{1}{\ell} \sum_{j=1}^\ell Cf(x,v_j) > \kappa(x) + \eta'.
$$ 
As above, we say that $f$ solves \eqref{eq_fsoluz} in GW-sense on $\Omega$ if it solves it at every $x_0 \in \Omega$. Clearly, if $f \in C^2(\Omega)$ then \eqref{eq_fsoluz} is satisfied in the pointwise sense on $\Omega$. We next define the following subsheafs $\mathscr{S}^1, \mathscr{S}^2$ and $\mathscr{S}$ by setting:
\[
	\begin{array}{lcl}
	\mathscr{S}_p^1 & = & \disp \Big\{ [f]_p \ : \ f \in \lip_\loc, \ |\nabla f| < \beta \ \text{in a neighbourhood of $p$}\Big\} \\[0.3cm]
	\mathscr{S}_p^2 & = & \disp \Big\{ [f]_p \ : \ f \in \lip_\loc, \ \PP_\ell^-[\nabla^2 f] > \kappa \ \text{in a neighbourhood of $p$}\Big\}.
	\end{array}.
	\] 	
and 
	\[
	\mathscr{S}_p = \mathscr{S}_p^1 \cap \mathscr{S}_p^2.
	\]
\begin{lemma}
$\mathscr{S}$ enjoys:
\begin{itemize}
\item[a)] the maximum closure property: If $[f_1]_p, [f_2]_p\in\mathscr{S}_p$, then $[\max\{f_1,f_2\}]_p\in\mathscr{S}_p;$
\item[b)] the $C^\infty-$stability property: If $K$ is a compact subset of $\Omega$ and $f\in\Gamma(\mathscr{S},\Omega)$, then there exists $\eps>0$ so that, if the $C^2-$norm of $\psi\in C^{\infty}(M)$ in $K$ is smaller than $\eps,$ then $[f + \psi]_p\in\mathscr{S}_p$ for all $p\in K$;
\item[c)] the local approximation property: For every $x\in \Omega$ there exists an open neighbourhood $V \subset \Omega$ of $x$ such that, for every $K \subset V$ compact, for every constant $\delta>0$ and for every $f\in\Gamma(\mathscr{S},V)$ that is $C^\infty$ in a (possibly empty) compact subset $K'$ of $K$, there exists $\tilde{f} \in C^\infty(V)$ such that $\tilde{f}\in\Gamma(\mathscr{S},V)$, $|\tilde{f}-f|<\delta$ on $K$, and $\tilde{f}$ is $\delta-$close to $f$ in the $C^\infty$-topology on $K'.$ 
\end{itemize}
\end{lemma}
\begin{proof}
The proof of this fact is an adaptation of the one given by \cite{GW,GW-1,Wu}. Items a) and b) follow immediately from the definition. The proof that $\mathscr{S}^1$ satisfies c) is essentially Lemma 8 in \cite{GW-1}, while a straightforward modification in the proof of Lemma 2 in \cite{Wu} proves that $\mathscr{S}^2$ satisfies c). For both of the sheaves $\mathscr{S}^1$ and $\mathscr{S}^2$, the local approximation property is achieved by means of Riemannian convolution, so the approximating functions $\tilde{f}$ in $\Gamma(\mathscr{S}^1,V)$ and $\Gamma(\mathscr{S}^2,V)$ can be chosen to be the same. Consequently, $\mathscr{S}$ satisfies c).
\end{proof}

As a direct consequence of the above lemma and Corollary 1 of Theorem 4.1 in \cite{GW}, we have the following smoothing theorem:

\begin{theorem}\label{C-infty-appro}
Given $f\in\Gamma(\mathscr{S},\Omega)$ and a positive continuous function $\xi$ on $\Omega$, there exists $\tilde{f} \in C^\infty(\Omega) \cap \Gamma(\mathscr{S},\Omega)$ such that $|\tilde{f}-f|<\xi.$
\end{theorem}

Our next goal is to prove that the function $u$ of Proposition \ref{prop_constr_u} lies in $\Gamma(\mathscr{S},\Omega_R).$ This is not obvious since, to our knowledge, the relations between solving \eqref{eq_fsoluz} in the barrier and GW sense have not been fully clarified. The problem has been addressed by H. Wu in \cite{Wu-complex,Wu}, and to apply his results, we need to assume a further condition on $\partial \Omega$.
 
\begin{definition}\label{def_locbouben}
An open subset $\Omega \subset M$ whose boundary satisfies 
	\begin{equation}\label{eq_locbou}
	\PP_{\ell-1}^-[\II_{\partial \Omega}] \ge \Lambda_{\ell-1}, \qquad \PP_\ell^-[\II_{\partial \Omega}] \ge \Lambda_\ell 	
	\end{equation}
in the barrier sense, for some continuous functions $\Lambda_{\ell-1}, \Lambda_\ell$ on $M$, is said to have \emph{locally bounded bending from outwards} if, for every compact set $A \subset \partial \Omega$ and $\eps>0$, there exists a constant $C_{A,\eps}$ such that every $y \in A$ admits a supporting hypersurface $S_y$ with 
	\begin{equation}\label{prop_locbouben}
	\begin{array}{c}
	\PP_{\ell-1}^-[\II_{S_y}](y) \ge \Lambda_{\ell-1}(y) - \eps, \qquad \PP_{\ell}^-[\II_{S_y}](y) \ge \Lambda_{\ell}(y) - \eps	\\[0.2cm]
	\text{and} \qquad - C_{A,\eps} \le \lambda_i(\II_{S_y}) \le C_{A,\eps} \quad \forall \, i \in \{1,\ldots, m-1\}.
	\end{array}
	\end{equation}
\end{definition}

\begin{remark}
\emph{Evidently, every boundary of class $C^2$ satisfying \eqref{eq_locbou} also has locally bounded bending from outwards, but the class is more general, as it includes, for instance, cones of the type 
	\[
	\Omega = \Big\{ x \in \R^m \ : \ \langle x ,v \rangle \ge |x| \cos \theta  \Big\},  
	\]
for fixed unit vector $v$ and angle $\theta \in (0, \pi/2)$. Indeed, $\II_{\partial \Omega}(y)$ has eigenvalues $0$ in direction $y$ and $(\cot \theta/|y|)$ on $y^\perp$. Hence, for every $\delta > 0$,  
	\begin{equation}\label{eq_locbou-cone}
		\PP_{\ell-1}^-[\II_{\partial \Omega}] \ge \frac{(\ell-2)\cot \theta}{|y|+ \delta}, \qquad \PP_\ell^-[\II_{\partial \Omega}] \ge \frac{(\ell-1)\cot\theta}{|y|+ \delta}	
	\end{equation}
in the barrier sense, and $\Omega$ has locally bounded bending from outwards, with constant $C_{A,\delta}$ both depending on $A$ and $\delta$.	
}
\end{remark}

The conclusion $u \in \Gamma(\mathscr{S},\Omega_R)$ will be a consequence of Wu's results and of the following compactness lemma: 

\begin{lemma}\label{lem_compactness}
Suppose that $\Omega$ satisfies \eqref{eq_locbou} and has locally bounded bending from outwards. Fix $x_0 \in \Omega$ and balls $B_{R_j}$, $j \in \{1,2\}$ centered at $x_0$ with radii $R_j$ such that
	\[
	R_1 < \mathrm{dist}(x_0, \partial \Omega), \qquad R_2 \ge 2R_1 + \mathrm{dist}(x_0, \partial \Omega). 
	\]
Then, for every $\eps_0 > 0$, there exists a constant $C$ depending on $\eps_0, R_1,R_2$, on the geometry of $B_{R_2}$, and on the constant $C_{A,\eps_0}$ of the set $A = \partial \Omega \cap \overline{B}_{R_2}$ guaranteed by the locally bounded bending property, such that the following holds: for every $x \in B_{R_1}$ and every nearest point $y \in \partial \Omega$ to $x$, there exists a supporting hypersurface $S_y$ at $y$ such that 
	\begin{equation}\label{prprp}
	\PP_{\ell-1}^-[\II_{S_y}](y) \ge \Lambda_{\ell-1}(y) - \eps_0, \qquad \PP_{\ell}^-[\II_{S_y}](y) \ge \Lambda_{\ell}(y) - \eps_0	
	\end{equation}
and whose corresponding distance $r_y = \mathrm{dist}(S_y, \cdot)$ satisfies
	\begin{equation}\label{twobounds_hessian}
	x \not \in \cut(S_y), \qquad  -C \le \nabla^2 r_y(x) \le C.
	\end{equation} 
\end{lemma}

\begin{proof}
Our restrictions on $R_1,R_2$ only serve to guarantee that every nearest point in $\partial \Omega$ to a point in $B_{R_1}$ lies in $B_{R_2}$. Let $\mathscr{F}$ be the family of triples $(\gamma, S,J)$, with $\gamma : [0, r(x)] \ra M$ a segment from $\partial \Omega$ to a point $x \in B_{R_1}$, $S$ a supporting hypersurface at $\gamma(0)$ that satisfies \eqref{prop_locbouben}, and $J$ the Jacobi tensor along $\gamma$ subjected to the initial conditions $J(0) = I$, $J'(0) = - \II_S$. Note that, for each $v$ parallel, $Jv$ generates a variation that is tangent to $S$ at time $0$ (we shortly say that $J$ issues from $S$). Consider a positive $\eps << \eps_0$ to be chosen later, only depending on $\eps_0, R_1,R_2, C_{A, \eps_0}, \overline{B}_{R_2}$. To each $(\gamma,S,J) \in \mathscr{F}$ we associate a triple $(\gamma, S^\eps, J^\eps)$, with $S^\eps$ be obtained by bending $S$ outwards by a factor $\eps$, as in the proof of Lemma \ref{lem_noncut}, and $J^\eps$ be the Jacobi tensor issuing from $S^\eps$, whose initial derivative is $-\II_{S}+\eps I$. By Lemma \ref{lem_noncut}, the ending point $x$ of $\gamma$ is not in the cut-locus of $S$, so the associated distance $r_\eps$ is smooth near $x$ and $J^\eps$ is invertible on $[0, r(x)]$. We claim that there exist $C_1,C_2$ depending on $\overline{B}_{R_2}, \eps, C_{A,\epsilon_0}$ such that
	\begin{equation}\label{uplower}
	\begin{array}{l}
	C_1 \le |J^\eps z| \le C_2 \qquad \text{on } \, [0, r(x)], \\[0.2cm]
	\text{for every } \, (\gamma,S,J) \in \mathscr{F} \ \text{ and unit field $z \in (\gamma')^\perp$ parallel along $\gamma$}.
	\end{array}
	\end{equation}
This and the Jacobi equation imply $|(J^\eps)'z| \le C_3$ for some constant $C_3$ with the same dependences as $C_1,C_2$, so the identity
	\[
	\nabla^2 r_\eps\left( \frac{J^\eps z}{|J^\eps z|},\frac{J^\eps z}{|J^\eps z|} \right) = \frac{1}{2} \frac{\di}{\di t} \log |J^\eps z|^2 
	\]
implies the uniform boundedness of $\nabla^2 r_\eps$. The thesis follows, up to choosing $S^\eps$ to be the desired supporting hypersurfaces (indeed, up to replacing $\eps_0$ with, say, $2\eps_0$ in \eqref{prprp}). Setting 
	\[
	c^2_R = \sup_{B_{R_2}} |\Sect|,
	\]
by Rauch comparison theorem
	\[
	|J^\eps z| \le \cosh(c_R r(x)) + \frac{C_{A,\eps_0} + \eps}{c_R}\sinh(c_R r(x)) \qquad \text{on $[0, r(x)$]}, 
	\]
so the upper bound in \eqref{uplower} directly follows from $r(x) \le R_2$. To prove the lower bound, we proceed by contradiction assuming the existence of $(\gamma_j,S_j,J_j) \in \mathscr{F}$, of unit vector fields $z_j$ parallel along $\gamma_j: [0,r(x_j)] \ra M$, and of $T_j \le r(x_j)$, such that $|J_j^\eps(T_j) z_j| \ra 0$ as $j \ra \infty$. Let $x_j \in B_{R_1}$ be the ending point of $\gamma_j$. Up to subsequences, $\gamma_j \ra \gamma$ for some segment $\gamma$ from $y \in \partial \Omega$ to $x \in \overline{B}_{R_1}$, $z_j \ra z$ and $T_j \ra T$. Furthermore, because of the locally bounded bending property, $\II_{S_j} \ra B$ for some $B$ with $- C_{A,\eps_0} I \le B \le C_{A,\eps_0} I$. Therefore, $J_j \ra J$ and $J_j^\eps \ra J^\eps$ in $C^1_\loc(\R)$, with $J, J^\eps$ the Jacobi tensors on $\gamma$ with initial derivatives $-B$ and $-B+\eps I$, and the convergence is thought to be component-wise in $C^1_\loc$ with respect to parallel frames on $(\gamma_j')^\perp$ smoothly converging to a parallel frame for $(\gamma')^\perp$. Hereafter, we identify tensors with their matrix representations in these fixed parallel frames, so for instance $J_j,J : \R \ra \mathfrak{gl}(m-1)$. Passing to the limit, $|J^\eps(T) z|=0$. As in Lemma \ref{lem_noncut}, $\bar V_\eps(t) = J^\eps(t)e^{-\eps t}z$ is a vector field along $\gamma$ that has initial velocity $-Bz$, vanishes at $\gamma(T)$ and (when replaced by zero on $[T, r(x)]$) generates a variation that decreases lengths and fixes $x$. In particular, referring to \eqref{eq_impoperdopo}, 
	\[
	\partial^2_s E(f_s)(0) \le - \eps + \eps^2 C_2^2 R_2 < - \frac{\eps}{2},
	\]
if $\eps$ is small enough. This would lead to the sought contradiction with the fact that $\gamma$ is a segment, provided that we guarantee that the variation $f_s$ can be constructed to issue from a supporting hypersurface at $y$ (with second fundamental form $B$). Although this should not be difficult, we prefer to reach a contradiction by using a variation of $\gamma_j$ for large enough $j$, as the already know the existence of the supporting hypersurface $S_j$. Let therefore $\psi : \R \ra [0,1]$ be a cut-off function with $\psi \equiv 1$ on $\left[0, \frac{T}{2}\right)$ and ${\rm spt } \psi \subset \left[0, \frac{3T}{4}\right)$. Reparametrizing according to $\tau_j(t) = \psi(t) t + (1-\psi(t))\frac{t T}{T_j}$, by continuity the field
	\[
	V_{\eps,j}(t) = \left\{ \begin{array}{ll}
	\disp \Big( (1-\psi)J^\eps + \psi J^\eps_j \Big)(\tau_j(t))e^{-\eps t}z & \quad \text{if } \, t \in [0,T_j], \\[0.3cm]
	0 & \quad \text{if } \, t \in (T_j, \infty)
	\end{array}\right.
	\]
transplanted on $\gamma_j$, generates a variation $f_{s,j}$ that is tangent to $S_j$ (because $V_{\eps, j}'(0) = - \II_{S^\eps_j}z + \eps z = -\II_{S_j}z$), fixes $x_j$ and still satisfies 
	\[
	\partial^2_s E(f_{s,j})(0) = \partial^2_s E(f_{s})(0) + o_j(1) < - \frac{\eps}{2},
	\]
for $j$ large enough, contradicting the fact that $\gamma_j$ is a segment on $[0, r(x_j)]$.
\end{proof}

\begin{proposition}\label{u-good}
In the assumptions of Proposition \ref{prop_constr_u}, if $\partial \Omega$ has locally bounded bending from outwards then $u \in\Gamma(\mathscr{S},\Omega_R)$ with the choices 
	\[
	\beta > C_2 u \ \ \text{on } \, \Omega_R, \qquad  \qquad \kappa(x) = h C_2 u(x) + \left\{ \begin{array}{ll}
	\bar \delta - \eps & \quad \text{if } \, h < \Lambda_\ell, \\[0.2cm]
	- \eps & \quad \text{if } \, h = \Lambda_\ell,
	\end{array}\right. 
	\]
where $\eps$ is any given positive constant. Consequently, for every $0 < t < s \le R$ and $\eps>0$, there exists a $C^\infty$ function $\bar u \in \Gamma(\mathscr{S},\Omega_R)$ such that
\begin{equation}\label{proprie_ut}
\left\{ \begin{array}{ll}
\PP_\ell^-[\nabla^2 \bar u] - h |\nabla \bar u| >  \left\{ \begin{array}{ll}
\bar \delta/2 & \quad \text{on $\Omega_R$, if $h < \Lambda_\ell$,} \\[0.3cm]
- \eps & \quad \text{on $\Omega_R$, if $h = \Lambda_\ell$,}
\end{array}\right. \\[0.6cm]
\disp \limsup_{r(x) \ra s} \bar u(x) < \inf_{\{r \le t\}} \bar u < \sup_{\Omega_R} \bar u < \infty.
\end{array}\right.
\end{equation}
\end{proposition}
\begin{proof}
By construction, $u = \eta(r) \in \Gamma(\mathscr{S}^1,\Omega_R)$, so we just need to prove that $u\in \Gamma(\mathscr{S}^2,\Omega_R)$. We apply the compactness Lemma \ref{lem_compactness} on $\Omega$, with balls of suitably chosen radii $R_1,R_2$ so that $\overline{B}_{R_1}\subset \Omega_R$, to deduce that every $x \in B_{R_1}$ has a nearest point $y \in \partial \Omega$, and a supporting hypersurface $S_y$ at $y$, such that the distance $r_y = \mathrm{dist}(S_y, \cdot)$ satisfies 
	\[
	\PP_\ell^-[\nabla^2(\eta\circ r_y)](x) > \kappa(x), \qquad - C \le \nabla^2 (\eta\circ r_y)(x) \le C, 
	\]
for some uniform constant $C$. The two inequalities enable us to repeat verbatim the proof of Proposition 2 in \cite{Wu-complex} (cf. also Lemma 3 in \cite{Wu}) to deduce $u = \eta(r) \in \Gamma(\mathscr{S}^2,\Omega_R)$, as claimed.\par
Next, we recall that $u$ is strictly decreasing as a function of the distance from $\partial\Omega$, so let 
	\[
	\tau_{t,s} = \min \big\{ \eta(0)-\eta(t), \eta(t) - \eta(s)\big\} > 0, 
	\]
choose $0 < \xi, \beta \in C(\overline{\Omega}_R)$ satisfying
	\[
	\xi < \min \left\{ \frac{\tau_{t,s}}{4}, \frac{\bar \delta}{4 h},\right\}, \qquad 
	\beta = C_2 u + \xi
	\]
and let $\bar u \in \Gamma(\mathscr{S},\Omega_R)$ be the smooth approximation of $u$ guaranteed by Theorem \ref{C-infty-appro} with $|u- \bar u| < \xi$ on $\Omega_R$. If $h < \Lambda_\ell$ (the other case being analogous),  
	\[
	\PP_\ell^-[\nabla^2 \bar u] > \frac{3\bar \delta}{4} + h C_2 u(x) > \frac{3\bar \delta}{4} + h(|\nabla \bar u|  - \xi) > \frac{\bar \delta}{2} + h|\nabla \bar u|
	\]
	on $\Omega_R$. Moreover, $|\bar u - u| < \xi < \tau_{t,s}/4$ on $\Omega_R$, that readily implies \eqref{proprie_ut}.
\end{proof}

\section{Maximum principles at infinity for varifolds}\label{sec_maxprinc}

This section is devoted to prove a maximum principle at infinity, and a parabolicity criterion, for varifolds in a complete Riemannian manifold. Our results adapt, to the varifold setting, the proofs of parabolicity and weak maximum principle at infinity via integral estimates obtained, respectively, in Theorems 5.1 and 4.1 of \cite{PRS1}.\\
Let us first recall some basic facts about varifolds, following \cite{simon-book}: let $V$ be an $\ell$-dimensional varifold in $M$, that is, a Radon measure on the Grassmannian $G_\ell(M)$ of $\ell$-planes on $M$. Given a $C^1$ vector field $Z$ compactly supported in an open set $\Omega\Subset M$, the first variation is defined as
	\[
	\delta V(Z) : = \disp \left.\frac{\di}{\di t}\right|_{t=0} ((\Phi_t)_\sharp V)(\Omega) = \int_{G_\ell(\Omega)}{\rm div}^{\mcal{W}}Z {\rm d}V(p,\mcal{W}),
	\]
where $\Phi_t : G_\ell(\Omega) \ra G_\ell(\Omega)$ is induced by the flow of $Z$ in the obvious way, and 
	\[
	{\rm div}^{{\mcal W}}Z =\sum_{i=1}^{\ell}g(\nabla_{e_i}Z,e_i), \qquad \text{with $\{e_i\}$ an orthonormal basis of ${\mcal W}$}. 
	\]
If $V$ has locally bounded first variation, i.e, 
	\[	
|\delta V(Z)|\leq C\sup_M|Z| \qquad  \text{for all $Z$ compactly supported on $\Omega$,}
	\]
then the total variation measure $\|\delta V\|$ is a Radon measure on $M$, where $\|\delta V\|$ is characterized by
\[
\|\delta V\|(\Omega)=\sup_{Z,|Z|\leq1,{\rm spt} Z\Subset\Omega}|\delta V|(Z).
\]
Splitting $\|\delta V\|$ into its absolutely continuous part and its singular part $\sigma$ with respect to the weight measure $\|V\| = \pi_{\sharp} V,$ where  $\pi : G_\ell(M) \ra M$ is the canonical projection, one gets 
	\[
\int_{G_\ell(\Omega)}{\rm div}^{\mcal{W}}Z {\rm d}V(p,\mcal{W})= - \ell \int_{\Omega}\langle {\bf H},Z\rangle{\rm d}\|V\|-\int_{{\rm spt } \sigma}\langle\nu,Z\rangle{\rm d}\sigma.
	\]
We call the vector field ${\bf H} \in L^1_\loc (M, \|V\|)$ the (normalized) mean curvature of $V$, ${\rm spt }\,  \sigma$ the generalized boundary of $V,$ and $\nu: {\rm spt } \, \sigma \ra \mathbb{S}^{m-1}$ the unit co-normal of $V.$ For notational convenience, it is customary to denote with $\|\partial V\|$ the measure $\sigma$, to keep track of the fact that $\sigma$ is a boundary measure related to $V$.\par
Given a varifold $V$ with locally bounded first variation, and given $0 \le h \in L^1_\loc(M, \|V\|)$, observe that the condition
\begin{equation}\label{bounded varifold}
\delta V(Z)+ \ell \int_{M}h|Z|{\rm d} \|V\|\geq 0 \qquad \forall \, Z \ \text{ compactly supported on $M\backslash {\rm spt} \|\partial V\|$} 
\end{equation}
	is equivalent to say that 
	\[
	|{\bf H}| \le h \qquad \text{$\|V\|$-a.e.} \, .
	\]
A $\ell$-dimensional varifold $V$ is called rectifiable if there exists a countably $\ell$-rectifiable set $\Sigma \subset M$ and a function $0 < \theta \in L^1_\loc(\mathscr{H}^\ell \measrest \Sigma)$ that is $\mathscr{H}^\ell$-a.e. positive on $\Sigma$, so that 
\[
V(U)=\int_{\pi(U)\cap \Sigma}\theta(p) {\rm d}\mathscr{H}^\ell(p) \qquad \forall \, U \subset G_\ell(\Omega).
\]
In this case, we write $V=V(\Sigma,\theta)$.	

\begin{theorem}[\textbf{Maximum principle at infinity}]\label{MPV}
Let $(M,g)$ be a complete Riemannian manifold, and suppose that $V$ is a $\ell$-dimensional varifold with locally bounded first variation and normalized mean curvature vector ${\bf H}$ satisfying 
	\[
	|{\bf H}| \le h \qquad \text{$\|V\|$-a.e.}, 
	\]
for some $0 \le h \in L^1_\loc(M, \|V\|)$. Let $u:M\to\R$ be a $C^2$ function so that
\begin{equation}\label{u. cres.}
\hat{u}=\limsup_{p \in {\rm spt} \|V\|, \, r(p)\to\infty}\frac{u(p)}{r(p)^\sigma}< \infty,
\end{equation}
for some constant $\sigma \in [0,2]$, where $r$ is the distance in $M$ from a fixed origin. Admit that for some $\gamma\in\mathbb{R}$ we have 
	\[
	{\rm spt} \|V\|\cap \Omega_\gamma\neq\varnothing, \qquad {\rm spt} \| \partial V \|\cap \Omega_\gamma = \varnothing, 
	\]
where $\Omega_\gamma= \{u> \gamma\}$. Let $\alpha \in \R$ and assume that either
	\begin{eqnarray}
	\alpha < 2-\sigma & \quad \text{and} & \quad \liminf_{r \ra \infty} \frac{\log \|V\|(B_r)}{r^{2-\sigma-\alpha}} := d_0 <\infty, \qquad \text{or} \label{Vol.expo} \\
	\alpha = 2-\sigma & \quad \text{and} & \quad \liminf_{r \ra \infty} \frac{\log \|V\|(B_r)}{\log r} := d_0 <\infty. \label{Vol.poli}
	\end{eqnarray}
Then, 
\begin{equation}\label{main. est.-1}
\disp \|V\|\text{-}{\rm ess}\inf_{\Omega_\gamma} \left\{ [1+r]^\alpha \Big[ \PP_\ell^-[\nabla^2 u]-h|\nabla u|\Big] \right\} \leq \frac{C(\sigma,\alpha,d_0)}{\ell} \max\{\hat{u},0\},
\end{equation}
where, setting
	\[
	{\mathcal I}_{u,\gamma}(V) := \int_{G_\ell(\Omega_\gamma)} |\nabla^{\mathcal W} u|^2 \di V(p,{\mathcal W}), 
	\]
the constant $C(\sigma,\alpha,d_0)$ is defined as follows: 
\begin{enumerate}
\item[a.] if ${\mathcal I}_{u,\gamma}(V) = 0$, then $C=0$; 
\item[b.] if ${\mathcal I}_{u,\gamma}(V) > 0$ and $\alpha<2-\sigma$, 
\[
C(\sigma,\alpha,d_0):=\left\{
\begin{array}{ccccc}
0 &{\rm if}& \sigma=0 \\[0.2cm]
d_0(2-\sigma-\alpha)^2 &{\rm if}& \sigma > 0, & \alpha<2(1-\sigma) \\[0.2cm]
d_0\sigma(2-\sigma-\alpha) &{\rm if}& \sigma > 0, & \alpha\geq 2(1-\sigma);
\end{array}
\right.
\]
\item[b.] if ${\mathcal I}_{u,\gamma}(V) > 0$ and $\alpha=2-\sigma$,
\[
C(\sigma,\alpha,d_0):=\left\{
\begin{array}{ccccc}
\sigma(\sigma+d_0-2) &{\rm if}& \sigma+d_0\geq 2 \\[0.2cm]
0 &{\rm if}& \sigma+d_0<2. \\
\end{array}
\right.\]
\end{enumerate}
\end{theorem}

\begin{proof}
Fix a constant $b>\max\{\hat{u},0\}.$ From \eqref{u. cres.} we infer that there exists $\nu \in \R$ so that
\begin{equation}\label{u boun.}
\frac{u+\nu}{[1+r]^\sigma}<b \ \ \text{on } \, {\rm spt } \|V\|, \qquad u(p_0)+\nu>0 \ \ \text{for some } \,  p_0\in {\rm spt}\|V\|.
\end{equation}
Consequently, up to replacing $u$ with $u+\nu$, we can suppose that \eqref{u boun.} holds for $u.$

Next, observe that once \eqref{main. est.-1} holds for some $\gamma'$, then it holds for every $\gamma''\leq\gamma'$. In particular, up to increasing $\gamma$ we may assume $\gamma \ge 0$. Define
\[
K \ := \ \|V\|\text{-}{\rm ess}\inf_{\Omega_\gamma} \left\{ [1+r]^\alpha \Big[ \PP_\ell^-[\nabla^2 u]-h|\nabla u|\Big] \right\}.
\]
If $K\leq0$, then the thesis follows at once. So, let us assume that $K>0$ and observe that, for $V$-a.e.  $(p,{\mcal W}) \in G_\ell(\Omega_\gamma\cap{\rm spt} \|V\|)$, 
\begin{equation}\label{div. est.}
\left({\rm div}^{{\mcal W}} \nabla u\right)(p)\geq \frac{\ell K}{[1+r(p)]^\alpha}+ \ell h|\nabla u|(p).
\end{equation}

Let $\psi \in C^\infty_c(M)$ be a cut-off function to be chosen later, and given a small $\eps>0$, consider two functions $\lambda:\R\to\R$ and $F:\R^2\to\R$ satisfying
\begin{equation}\label{def_lam_WMP}
\begin{array}{l}
0\leq\lambda\leq1, \quad \lambda\equiv0 \ {\rm on}\ (-\infty,\gamma + \eps/2], \quad \lambda\equiv1 \ {\rm on}\ [\gamma+\eps,\infty), \\[0.2cm]
\lambda>0 \ {\rm and} \ \lambda'\geq0 \ {\rm on}\ (\gamma,\infty)
\end{array}
\end{equation}
and 
\[
F>0, \ \ \  \frac{\partial F}{\partial v}(v,r)<0 \qquad \text{on } \,  [0,\infty)\times[0,\infty).
\]
Set
	\[
	v=\beta[1+r(p)]^\sigma-u, \qquad \beta>b,
	\]
and define the vector field $Z=-\psi^2\lambda(u)F(v,r)\nabla u$, that by construction is compactly supported in $\Omega_\gamma.$ Note that
\begin{equation}\label{v est.}
(\beta-b)[1+r]^\sigma\leq v\leq \beta[1+r]^\sigma \qquad \text{on } \,  \Omega_\gamma \cap {\rm spt } \|V\|.
\end{equation} 
Using \eqref{div. est.}, a straightforward computation gives for $V$-a.e. $(p, {\mcal W}) \in G_\ell(\Omega_\gamma\cap{\rm spt} \|V\|)$,
\begin{eqnarray}\label{div w est.}
\nonumber{\rm div}^{{\mcal W}}Z&=&-\psi^2\lambda F{\rm div}^{{\mcal W}}\nabla u \\
& & -g\left(\nabla^{{\mcal W}} u,2\psi\lambda F\nabla^{{\mcal W}} \psi+\psi^2\lambda'F\nabla^{{\mcal W}}u+\psi^2\lambda\left[\frac{\partial F}{\partial v}\nabla^{{\mcal W}}v+\frac{\partial F}{\partial r}\nabla^{{\mcal W}}r\right]\right)\\
&\leq&2\psi\lambda F|\nabla^{{\mcal W}}\psi||\nabla^{{\mcal W}} u|-\psi^2\lambda\left|\frac{\partial F}{\partial v}\right| B- \ell h|Z|,
\end{eqnarray}
where
\begin{equation}\label{B def.}
B=\frac{F}{\left|\frac{\partial F}{\partial v}\right|} \ell K[1+r]^{-\alpha}+|\nabla^{{\mcal W}} u|^2+\left[\frac{\frac{\partial F}{\partial r}}{\left|\frac{\partial F}{\partial v}\right|}-\beta\sigma[1+r]^{\sigma-1}\right]g(\nabla^{{\mcal W}} r,\nabla^{{\mcal W}} u).
\end{equation}

We first examine the case ${\mathcal I}_{u,\gamma}(V) = 0$: first, notice that $|\nabla^{{\mcal W}} u|(p)=0$ for all $(p,{\mcal W})\in{\rm spt} V$ (here ${\rm spt} V$ denotes the support of $V$ as measure in $G_\ell(\Omega_\gamma)$). In particular,  \eqref{div w est.} and 
\eqref{B def.} together imply that
\[
\frac{1}{\ell} {\rm div}^{{\mcal W}} Z \leq -\psi^2\lambda FK[1+r]^{-\alpha}-h|Z| \qquad \text{for $V$-a.e. } \, (p, {\mcal W}) \in{\rm spt} V.
\]
Therefore, integrating and using \eqref{bounded varifold} we infer
\[
K\int_{\Omega_\gamma}F\psi^2\lambda[1+r]^{-\alpha}{\rm d}\|V\|\leq 0.
\]
The arbitrariness of $\psi, \lambda$ implies $K\leq 0$, as desired. We hereafter assume that ${\mathcal I}_{u,\gamma}(V) > 0$, and fix $R_0>0$ such that
 	\begin{equation}\label{eq_R0}
 	\int_{G_\ell(\Omega_\gamma \cap B_{R_0})}|\nabla^{{\mcal W}} u|^2 \di V(p, {\mcal W}) > 0.
 	\end{equation}
For $R>2R_0$ and $\theta\in(1/2,1)$, let $\psi:M\to\R$ be a cut off function so that
\[
0\leq\psi\leq1,\ \psi\equiv1\ {\rm in}\ B_{\theta R},\ \psi\equiv0\ {\rm in}\ M\setminus B_R\ {\rm and}\ |\nabla \psi|\leq \frac{2}{(1-\theta) R},
\]

We split the proof into the following cases.
\begin{flushleft}
{\it Case i:}\ $\sigma>0$, $\eta=\alpha+2(\sigma-1)<0$.
\end{flushleft}
We define $F(v,r)=\exp(-qv[1+r]^{-\eta})$, where $q>0$ is a constant that will be defined later. From 
\[
\frac{\frac{\partial F}{\partial r}}{\left|\frac{\partial F}{\partial v}\right|}=\frac{\eta v}{1+r},
\]
we infer 
\[
0\geq\frac{\frac{\partial F}{\partial r}}{\left|\frac{\partial F}{\partial v}\right|}-\beta\sigma[1+r]^{\sigma-1}\geq-\beta(\sigma-\eta)[1+r]^{\sigma-1}=-\beta(2-\sigma-\alpha)[1+r]^{\sigma-1}.
\]
Using \eqref{B def.}, we get
\begin{eqnarray}
\nonumber B&\geq&\frac{\ell K}{q}[1+r]^{\eta-\alpha}+|\nabla^{{\mcal W}} u|^2-\beta(2-\sigma-\alpha)[1+r]^{\sigma-1}|\nabla^{{\mcal W}} u|\\
&=&\frac{\ell K}{q}[1+r]^{2(\sigma-1)}+|\nabla^{{\mcal W}} u|^2-\beta(2-\sigma-\alpha)[1+r]^{\sigma-1}|\nabla^{{\mcal W}} u|.
\end{eqnarray}
In turn out that $B\geq \Lambda|\nabla^{{\mcal W}} u|^2$ provided that
\[
0<\Lambda\leq 1-q\frac{\beta^2(2-\sigma-\alpha)^2}{4 \ell K}.
\]
So, if we take $\tau\in(0,1)$ and define $q=\tau \frac{4 \ell K}{\beta^2(2-\sigma-\alpha)^2}$, then the previous inequality is true for $\Lambda=1-\tau$. Using this fact and that $\left|\frac{\partial F}{\partial v}\right|=q[1+r]^{-\eta}F$, we conclude from \eqref{div w est.} that
\begin{eqnarray}
{\rm div}^{{\mcal W}} Z\leq 2\psi\lambda F|\nabla^{{\mcal W}} \psi||\nabla^{{\mcal W}} u|-\Lambda q[1+r]^{-\eta}F|\nabla^{{\mcal W}} u|^2\psi^2\lambda- \ell h|Z|. 
\end{eqnarray}
Using \eqref{bounded varifold} with such $Z$, it follows
\begin{equation}\label{main est.}
\frac{\Lambda q}{2}\int_{G_\ell(\Omega_\gamma)}\psi^2\lambda F[1+r]^{-\eta}|\nabla^{{\mcal W}} u|^2{\rm d}V(p,{\mcal W})\leq\int_{G_\ell(\Omega_\gamma)}\psi\lambda F|\nabla^{{\mcal W}} \psi||\nabla^{{\mcal W}} u|{\rm d} V(p,{\mcal W}).
\end{equation}
Our choice of $\lambda$ and $\psi$ imply
\begin{equation*}
\int_{G_\ell(\Omega_\gamma)}\psi^2 F[1+r]^{-\eta}|\nabla^{{\mcal W}} u|^2{\rm d}V(p,{\mcal W}) \ge \int_{G_\ell(\Omega_{\gamma+\eps}\cap B_{R_0})}F[1+r]^{-\eta}|\nabla^{{\mcal W}} u|^2{\rm d}V(p,{\mcal W}).
\end{equation*}
Thus by Fatou's lemma and \eqref{v est.}, \eqref{eq_R0}, 
\begin{equation*}
0<\int_{G_\ell(\Omega_{\gamma}\cap B_{R_0})}F[1+r]^{-\eta}|\nabla^{{\mcal W}} u|^2{\rm d}V(p,{\mcal W})\leq\liminf_{\eps\to0}\int_{G_\ell(\Omega_{\gamma+\eps}\cap B_{R_0})}F[1+r]^{-\eta}|\nabla^{{\mcal W}} u|^2{\rm d}V(p,{\mcal W}).
\end{equation*}
This inequality ensures that the left hand side of \eqref{main est.} is uniformly positive for $R \ge R_0$. On the other hand, by Holder inequality
\begin{eqnarray*}
\int_{G_\ell(\Omega_\gamma)}\psi\lambda F|\nabla^{{\mcal W}} \psi||\nabla^{{\mcal W}} u|{\rm d} V(p,{\mcal W})&\leq&\left(\int_{G_\ell(\Omega_\gamma)}\psi^2\lambda F[1+r]^{-\eta}|\nabla^{{\mcal W}} u|^2{\rm d} V(p,{\mcal W})\right)^{\frac{1}{2}}\\
&\cdot&\left(\int_{G_\ell(\Omega_\gamma)}\lambda F[1+r]^\eta|\nabla^{{\mcal W}} \psi|^2{\rm d} V(p,{\mcal W})\right)^{\frac{1}{2}}.
\end{eqnarray*}
Substituting this into \eqref{main est.}, using $\lambda\leq 1$ and again Fatou's lemma, one gets
\begin{equation}\label{pri. est.}
0<E :=\left(\frac{\Lambda q}{2}\right)^2\int_{G_\ell(\Omega_\gamma\cap B_{R_0})}F|\nabla^{{\mcal W}} u|^2{\rm d}V(p,{\mcal W})\leq\int_{\Omega_\gamma} F[1+r]^\eta|\nabla \psi|^2{\rm d}\|V\|.
\end{equation}
From \eqref{Vol.expo}, for every $d>d_0$ there exists an increasing sequence $\{R_l\}$ with $R_1>2R_0$ so that
\begin{equation}\label{Area-est}
\|V\|(B_{R_l})\leq\exp(dR_l^{2-\sigma-\alpha}).
\end{equation}

Using this inequality and \eqref{v est.}, one obtains
\begin{eqnarray*}
\nonumber 0<E&\leq&\int_{{\rm spt}\|V\|\cap\Omega_\gamma\cap(B_{R_l}\setminus B_{\theta R_l})} F[1+r]^\eta|\nabla \psi|{\rm d}\|V\|\\
&\leq& 4\frac{[1+\theta R_l]^\eta R_l^{-2}}{(1-\theta)^2}\int_{{\rm spt} \|V\|\cap\Omega_\gamma\cap(B_{R_l}\setminus B_{\theta R_l})} F{\rm d}\|V\|\\
&\leq& 4\frac{[1+\theta R_l]^\eta R_l^{-2}}{(1-\theta)^2}\exp(dR_l^{2-\sigma-\alpha}-q(\beta-b)(1+\theta R_l)^{2-\sigma-\alpha}).
\end{eqnarray*}
The above inequality does not lead to contradictions when $l \ra \infty$ if and only if 
\[
d\geq q(\beta-b)\theta^{2-\sigma-\alpha}.
\]
Taking $\theta\to1$ we conclude that necessarily $d\geq q(\beta-b)$. Writing $\beta=tb$, using the definition of $q$, letting $\tau\to 1$ and isolating $K$ we get
\[
K\leq db\frac{(2-\sigma-\alpha)^2}{4 \ell}\frac{t^2}{t-1}.
\]
As the function $t^2/(t-1)$ attains a global minimum at $t=2$, letting $b\to\max\{\hat{u},0\}$ and $d\to d_0$ we conclude
\begin{equation}\label{Case I ine}
K\leq \frac{d_0}{\ell}\max\{\hat{u},0\}(2-\sigma-\alpha)^2
\end{equation}

\begin{flushleft}
{\it Case ii:}\ $\sigma=0$ (notice that $\eta=\alpha-2(1-\sigma)=\alpha-2<0$ by our hypothesis).
\end{flushleft}
We proceed exactly as in {\it Case i}. In fact, with little modification, we may conclude the validity of \eqref{Case I ine} for every bounded function $u$ as in the statement of the theorem, not necessarily positive. Since 
	\[
	[1+r]^\alpha \Big[\PP_\ell^-[\nabla^2 u]-h|\nabla u|\Big]=[1+r]^\alpha \Big[\PP_\ell^-[\nabla^2 (u-\hat{u})]-h|\nabla (u-\hat{u})|\Big],
	\] 
we thus easily deduce 
\[
K \ : = \ \|V\|\text{-}{\rm ess}\inf_{\Omega_\gamma} \left\{ [1+r]^\alpha \Big[ \PP_\ell^-[\nabla^2 u]-h|\nabla u|\Big] \right\} \le 0.
\]
\begin{flushleft}
{\it Case iii:}\ $\sigma>0$, $\eta=\alpha+2(\sigma-1)\geq0$.
\end{flushleft}
We choose $F(v,r)=\exp(-qv^{\frac{\sigma-\eta}{\sigma}})$, where $q$ is a constant which will be defined later. From
\[
\frac{F}{\left|\frac{\partial F}{\partial v}\right|}=\frac{\sigma}{2-\sigma-\alpha}\frac{v^{\frac{\eta}{\sigma}}}{q},
\]
we get
\begin{eqnarray}
B\geq|\nabla^{{\mcal W}} u|^2+\frac{\sigma}{2-\sigma-\alpha}\frac{\ell K}{q}(\beta-b)^\frac{\eta}{\sigma}[1+r]^{2(\sigma-1)}-\beta\sigma[1+r]^{\sigma-1}|\nabla^{{\mcal W}} u|
\end{eqnarray}
In turn out that $B\geq \Lambda|\nabla^{{\mcal W}} u|^2$ provided that
\[
0<\Lambda\leq1-q\frac{\beta^2\sigma(2-\sigma-\alpha)}{4 \ell K(\beta-b)^{\frac{\eta}{\sigma}}}.
\]
So, if we take $\tau\in(0,1)$ and $q=\tau\frac{4 \ell K(\beta-b)^{\frac{\eta}{\sigma}}}{\beta^2\sigma(2-\sigma-\alpha)}$, then the previous inequality is satisfied with $\Lambda=1-\tau$. Therefore, from \eqref{div w est.} and \eqref{v est.} we deduce
\[
{\rm div}^{{\mcal W}} Z\leq 2\lambda\psi F|\nabla^{{\mcal W}} \psi||\nabla^{{\mcal W}} u|-\Lambda (\beta-b)^{-\frac{\eta}{\sigma}} q\frac{(2-\sigma-\alpha)}{\sigma}\psi^2\lambda[1+r]^{-\eta}F |\nabla^{{\mcal W}} u|^2- \ell h|Z|.
\]
Following the argument of {\it Case i}, we obtain 
\begin{eqnarray}\label{A-inequality}
\nonumber0<E&:=&\left(\frac{\Lambda (\beta-b)^{\frac{-\eta}{\sigma}} q(2-\sigma-\alpha)}{2\sigma}\right)^2\int_{G_\ell(\Omega_\gamma\cap B_{R_0})}F|\nabla^{{\mcal W}} u|^2{\rm d}V(p,{\mcal W})\\
&\leq& 4[1+R_l][R_l(1-\theta)]^{-2}\exp(dR_l^{\sigma-\eta}-q(\beta-b)^{\frac{2-\alpha-\sigma}{\sigma}}(1+\theta R_l)^{\sigma-\eta}),
\end{eqnarray}
where $d>d_0$ and $\{R_l\}$ still satisfies \eqref{Area-est}, $R_1>2R_0$. Letting $l \ra \infty$ in \eqref{A-inequality} we deduce that necessarily
\[
d\geq q(\beta-b)^{\frac{2-\alpha-\sigma}{\sigma}}\theta^{\sigma-\eta}.
\]
Letting $\theta\to1$, using the expression of $q$, isolating $K$ and noting that $\eta+2-\sigma-\alpha=\sigma$, we obtain
\[
K\leq\frac{1}{\tau} d\frac{\sigma(2-\sigma-\alpha)\beta^2}{4 \ell (\beta-b)}.
\]
To conclude the proof, call $\beta=tb\ (t>1)$, let $\tau\to1$, $d\to d_0$ and $b\to\max\{\hat{u},0\}$ and minimize the resulting expression in $t>1$ to get
\[
K\leq \sigma(2-\sigma-\alpha)\frac{d_0}{\ell}\max\{\hat{u},0\}.
\]

\begin{flushleft}
{\it Case iv:} \(\alpha=2-\sigma\).
\end{flushleft}
We choose  $F(v,r)=F(v)=v^{-q}$, where $q$ will be defined later. From 
\begin{eqnarray*}
B\geq \frac{\ell K}{q}[1+r]^{2(\sigma-1)}(\beta-b)+|\nabla^{{\mcal W}} u|^2-\beta\sigma[1+r]^{\sigma-1}|\nabla^{{\mcal W}} u|,
\end{eqnarray*}
the lower bound $B\geq \Lambda|\nabla^{{\mcal W}} u|^2$ holds provided that 
\[
\Lambda=1-\frac{1}{4}\frac{q}{\ell K}\frac{\beta^2\sigma^2}{(\beta-b)}.
\]
In particular, if $\tau\in(0,1)$, $q=\tau\frac{4(\beta-b)\ell K}{\beta^2\sigma^2}$ the bound is satisfied with $\Lambda=1-\tau.$ With these choices,
\[
{\rm div}^{{\mcal W}} Z\leq 2\psi\lambda F|\nabla^{{\mcal W}} \psi||\nabla^{{\mcal W}} u|-\psi^2\lambda qv^{-1}F\Lambda|\nabla^{{\mcal W}} u|^2- \ell h|Z|.
\]

Arguing as in the previous cases, using now that, for $d> d_0$, the inequality $\|V\|(B_{R_l})\leq R_l^d$ holds along some increasing sequence $\{R_l\}$ with $R_1>2R_0$, we infer
\[
\begin{array}{lcl}
0 & < & \disp E : =\left(\frac{\Lambda q}{2}\right)^2\int_{G_\ell(\Omega_\gamma\cap B_{R_0})}v^{-1} F{\rm d}V(p,{\mcal W}) \\[0.5cm]
& \leq & \disp \disp C(1-\theta)^2(\beta-b)
^{-q}[1+\theta R_l]^{-q\sigma}R_l^{-2}[1+R_l]^{\sigma}R_l^d,
\end{array}
\]
where $C$ is a constant that does not depend on $l$. Necessarily, $d-2+\sigma\geq q\sigma$. This is incompatible with $d_0+\sigma\leq2$, forcing $K\leq0$ in this case. On the other hand, if $d_0+\sigma> 2,$ we put $\beta=tb$ $(t>1)$, using the definition of $q$, letting $b\to\max\{\hat{u},0\}$ and $d\to d_0$, and minimizing in $t$, one gets
\[
K\leq \frac{1}{\ell} \sigma(\sigma+d_0-2)\max\{\hat{u},0\}.
\]
This completes the proof of the theorem.
\end{proof}

\begin{theorem}[\textbf{Parabolicity}]\label{MPV_para}
Let $(M,g)$ be a complete Riemannian manifold, and suppose that $V$ is a $\ell$-dimensional varifold with locally bounded first variation and normalized mean curvature vector ${\bf H}$ satisfying 
	\[
	|{\bf H}| \le h \qquad \text{$\|V\|$-a.e.}, 
	\]
for some $0 \le h \in L^1_\loc(M,\|V\|)$. Let $u \in \lip_\loc(M)$ satisfy 
	\[
	\sup_{{\rm spt } \|V\|} u < \infty,
	\]
	and assume that, for some $\gamma \in \R$, the upper level set $\Omega_\gamma = \{u> \gamma\}$ satisfies  
	\[
	{\rm spt} \|V\| \cap \Omega_\gamma \neq \varnothing, \qquad {\rm spt} \|\partial V\| \cap \Omega_\gamma = \varnothing.
	\]
Let $\{u_\eps\}_\eps \subset C^2(M)$ be a locally equi-Lipschitz sequence of functions, converging to $u$ locally uniformly on $M$ as $\eps \to 0$ and satisfying 
	\[
	\PP_\ell^-[\nabla^2 u_\eps](p)-h|\nabla u_\eps|(p) \ge -\eps \qquad \text{for $\|V\|$-a.e. } \, p \in {\rm spt} \|V\| \cap \Omega_\gamma.
	\]
If
	\begin{equation}\label{cond_para}
\text{$V$ is rectifiable and } \qquad \int^{\infty} \frac{r \di r}{\|V\|(B_r)} = \infty,
\end{equation}
then $u$ is locally constant on ${\rm spt}\|V\| \cap \Omega_\gamma$.
\end{theorem}		
\begin{proof}
Up to renaming the sequence, we can assume that $\|u_\eps-u\|_\infty < \eps$. Fix $2\eps_0 \in (0, u^*-\gamma)$, choose $\gamma' = \gamma + \eps_0$ and assume $4\eps \in (0,\eps_0)$. Since $u_\eps^* > u^* - \eps > \gamma'$, observe that $\{u_\eps \ge \gamma'\} \subset \Omega_\gamma$ and is non-empty. Consider the vector field 
	\[
	Z_\eps = - \psi^2 \lambda(u_\eps) e^{u_\eps} \nabla u_\eps,
	\]
where $\psi$ is a cut-off to be chosen later, and $\lambda \in C^1(\R)$ satisfying
	\[
0\leq\lambda\leq1, \quad \lambda' \ge 0 \ \text{ on } \, \R, \quad \lambda\equiv0 \ \text{ on } \,  (-\infty,\gamma'], \quad \lambda > 0 \  \text{ on } \, (\gamma',\infty).
	\]
Note that $Z_\eps$ is compactly supported in $\Omega_\gamma$. For $V$-a.e. $(p,{\mcal W}) \in G_\ell(\Omega_\gamma)$, it holds
	\[	
	\begin{array}{lcl}
	{\rm div}^{{\mcal W}} Z_\eps & \le & \disp -2\psi \lambda(u_\eps) e^{u_\eps} \langle \nabla^{{\mcal W}} \psi, \nabla u_\eps \rangle - \psi^2 \lambda(u_\eps) e^{u_\eps} |\nabla^{{\mcal W}} u_\eps|^2 - \psi^2 \lambda(u_\eps) e^{u_\eps} {\rm div}^{{\mcal W}} \nabla u_\eps \\[0.3cm]
	& \le & \disp 2\psi \lambda(u_\eps) e^{u_\eps} |\nabla^{{\mcal W}} \psi||\nabla^{{\mcal W}} u_\eps| - \psi^2 \lambda(u_\eps)e^{u_\eps} |\nabla^{{\mcal W}} u_\eps|^2 - \psi^2 \lambda(u_\eps) e^{u_\eps} h |\nabla u_\eps| + \eps \lambda(u_\eps)\psi^2 e^{u_\eps}. 
	\end{array}
	\]
Integrating and using \eqref{bounded varifold} together with H\"older inequality, we deduce
	\begin{equation}\label{inte_byparts}
	\begin{array}{l}
	\disp \int_{G_\ell(\Omega_\gamma)} \psi^2 \lambda(u_\eps) e^{u_\eps} |\nabla^{{\mcal W}} u_\eps|^2 \di V(p,{\mcal W}) \\[0.5cm]
\quad \le \disp 2 \int_{G_\ell(\Omega_\gamma)} \psi \lambda(u_\eps) e^{u_\eps} |\nabla^{{\mcal W}} \psi||\nabla^{{\mcal W}} u_\eps| \di V(p,{\mcal W}) + \eps \int_{\Omega_\gamma} \psi^2 \lambda(u_\eps) e^{u_\eps} \di \|V\|\\[0.5cm]	
\quad \le \disp 2 \left( \int_{\Omega_\gamma} \psi \lambda(u_\eps) e^{u_\eps} |\nabla \psi| \di \|V\| \right)^{\frac{1}{2}} \left( \int_{G_\ell(\Omega_\gamma)} \psi |\nabla \psi| \lambda(u_\eps) e^{u_\eps} |\nabla^{{\mcal W}} u_\eps|^2 \di V(p,{\mcal W})\right)^{\frac{1}{2}} \\[0.4cm]
\quad \disp + \eps \int_{\Omega_\gamma} \psi^2 \lambda(u_\eps) e^{u_\eps} \di \|V\|.
	\end{array}
	\end{equation}
Given $0 < r < R$, let $\psi \in C^\infty_c(B_R)$ satisfy
	\[
	0 \le \psi \le 1, \quad \psi \equiv 1 \ \ \text{ on } \, B_r, \qquad |\nabla \psi| \le \frac{2}{R-r}.	
	\]
Then, letting $\lambda \uparrow \chi_{(\gamma', \infty)}$ ($\chi$ the  indicator function) and defining
	\[
	I_\eps(t) : = \int_{G_\ell(B_t \cap \Omega_{\gamma'})} e^{u_\eps} |\nabla^{{\mcal W}} u_\eps|^2 \di V(p,{\mcal W}),
	\]
from \eqref{inte_byparts} and $u_\eps^* : = \sup_M u_\eps < u^* + 1$ we deduce
	\[
	I_\eps(r) \le 4 e^{\frac{u^*+1}{2}} \left( \frac{\|V\|(B_R) - \|V\|(B_r)}{R-r} \right)^\frac{1}{2} \left(\frac{I_\eps(R)-I_\eps(r)}{R-r}\right)^{\frac{1}{2}} + \eps e^{u^*+1} \|V\|(B_R).
	\]
Setting
	\[
	I(t) = \limsup_{\eps \ra 0} I_\eps(t).
	\]
We therefore deduce 
	\[
	I(r) \le 4 e^{\frac{u^*+1}{2}} \left( \frac{\|V\|(B_R) - \|V\|(B_r)}{R-r} \right)^\frac{1}{2}  \left(\frac{I(R)-I(r)}{R-r}\right)^{\frac{1}{2}}.
	\]
Letting $R \downarrow r$ and using the monotonicity of $I(t)$ and $\|V\|(B_t)$, we deduce for a.e. $r>0$ the inequality
	\[
	I(r)^2 \le 16e^{u^*+1} \big( \|V\|(B_r)\big)' I'(r).
	\]
We claim that $I \equiv 0$. Otherwise, there would exist $r_0>0$ such that $I(r) \ge I(r_0)>0$, and integrating on $[r_0,r]$ we deduce
	\begin{equation}\label{eq_contr_para}
	\frac{1}{I(r_0)} - \frac{1}{I(r)} \ge \int_{r_0}^r \frac{I'(s)\di s}{I(s)^2} \ge \frac{e^{-u^*-1}}{16} \int_{r_0}^r \frac{\di s}{\big( \|V\|(B_s)\big)'}.
	\end{equation}
By \cite[Prop.1.3]{rigolisetti},
	\[
	\int_{r_0}^r \frac{s-r_0}{\|V\|(B_s)} \di s \le 2 \int_{r_0}^r \frac{\di s}{(\|V\|(B_s))'},
	\]
that together with \eqref{cond_para} enables to deduce
	\[
	\int^{\infty} \frac{\di s}{(\|V\|(B_s))'} = \infty.
	\]
Letting $r \ra \infty$ in \eqref{eq_contr_para}, we reach a contradiction. From $I \equiv 0$ and letting $\gamma'\to\gamma$, one gets
	\begin{equation}\label{eq_bella}
	\liminf_{\eps \to 0}\int_{G_\ell(B_t \cap \Omega_{\gamma})}  |\nabla^{{\mcal W}} u_\eps|^2\di V(p,{\mathcal W})= 0.
	\end{equation}
We next use that $V$ is rectifiable, namely, that $V = V\left(\Sigma,\theta\right)$ for a $\ell$-dimensional countably rectifiable set $\Sigma$, with $0 < \theta \in L^1_\loc(\mathscr{H}^\ell \measrest \Sigma)$ that is $\mathscr{H}^\ell$-a.e. positive on $\Sigma.$ With this notation, \eqref{eq_bella} becomes
\begin{equation}\label{eq_bella-1}
		\liminf_{\eps \to 0}\int_{\Sigma \cap B_t}  |\nabla^{\Sigma} u_\eps|^2\theta {\rm d}\mathscr{H}^\ell= 0.
\end{equation}
Since $u_\eps$ and $u$ are Lipschitz, they are weakly differentiable in the sense of \cite{Men} on $V\llcorner (B_t \cap \Omega_{\gamma})$, and by \eqref{eq_bella-1} we conclude that $u$ must have zero weak derivative on $V\llcorner (B_t \cap \Omega_{\gamma})$. The proof of \cite[Thm. 8.34]{Men} shows that a weakly differentiable function on a varifold with zero weak derivative must take only countably many values. Thus the varifold $V\llcorner (B_t \cap \Omega_{\gamma})$ has some decomposition into components on which the weakly differentiable function $u$ is constant. As $u$ is continuous, then it is constant on each connected component of ${\rm spt} \|V\| \cap \Omega_\gamma\cap B_t.$ Since $t$ was arbitrary, one deduces that $u$ is constant on each connected component of ${\rm spt} \|V\| \cap \Omega_\gamma.$
\end{proof}

\section{Proof of Theorem \ref{omori-yau_vari} and Corollary \ref{cor_spectrum}}

\begin{proof}[of Theorem \ref{omori-yau_vari}] We first consider the case $ \|{\bf H}\|_\infty < \Lambda_\ell$. Assume by contradiction that \eqref{bonita} (respectively, \eqref{bonita_bis}) does not hold, under the validity of $(\mathbb{A})$ (respectively, $(\mathbb{B})$). Fix 
	\[
	\begin{array}{ll}
	R \in \left(  {\rm dist}\big({\rm spt} \|\Sigma\|, \partial \Omega\big), {\rm dist}\big({\rm spt} \|\partial \Sigma\|, \partial \Omega\big) \right) & \quad \text{in case } \, (\mathbb{A}), \\[0.4cm]
	R \in \left( {\rm dist}\big({\rm spt} \|\Sigma\|, \partial \Omega\big), \min\Big\{\frac{\Lambda_\ell - \|\mathbf{H}\|_\infty}{c_+}, \dist({\rm spt} \|\partial \Sigma\|, \partial \Omega)\Big\} \right) & \quad \text{in case } \, (\mathbb{B}),
	\end{array}
	\]
Consider the radial function $u$ guaranteed by Proposition \ref{prop_constr_u}, and in view of Proposition \ref{u-good}, for small $\eps>0$ choose a smooth approximation $\bar u = u_\eps \in C^\infty(\Omega_R)$ satisfying 
	\begin{equation}\label{prop_u_teo}
\left\{ \begin{array}{ll}
\disp \{u_\eps  > \gamma\} \cap {\rm spt} \|\Sigma\| \neq \varnothing, \ \ \text{ for some } \, \gamma > \limsup_{r(x) \ra R} u_\eps(x) \\[0.3cm]
\PP_\ell^-[\nabla^2 u_\eps] - \|{\bf H}\|_\infty |\nabla  u_\eps| > \frac{\bar \delta}{2} \qquad \text{on } \, \{u_\eps > \gamma\} \\[0.4cm]
\sup_{\Omega_R} u_\eps < \infty.
\end{array}\right.
	\end{equation}
In particular, 
	\[
	\disp \inf_{\{u_\eps > \gamma\} \cap {\rm spt} \|\Sigma\|} \Big\{ \PP_\ell^-[\nabla^2  u_\eps]- \|{\bf H}\|_\infty |\nabla u_\eps|\Big\} \ge \frac{\bar \delta}{2}.
	\]
Because of our growth assumptions on $\|\Sigma\|(B_r)$, we can therefore apply the maximum principle at infinity, Theorem \ref{MPV}, to deduce a contradiction.\par
If $\|{\bf H}\|_\infty = \Lambda_\ell$, having fixed $R$ as above in case $(\mathbb{A})$ and using Propositions \ref{prop_constr_u} and \ref{u-good}, there exists a function $u$ depending only on the distance $r$ to $\partial \Omega$, and a sequence $\{u_\eps\} \in C^\infty(\Omega_R)$ of equi-Lipschitz functions converging uniformly to $u$, such that 
	\[
\left\{ \begin{array}{ll}
\disp \{u_\eps > \gamma\} \cap {\rm spt} \|\Sigma\| \neq \varnothing, \ \ \text{ for some } \, \gamma > \limsup_{r(x) \ra R} u(x)  \\[0.3cm]
\PP_\ell^-[\nabla^2 u_\eps] - \|{\bf H}\|_\infty |\nabla  u_\eps| \ge -\eps \qquad \text{on } \, \{u > \gamma\}, \\[0.4cm]
\sup_{\Omega_R} u < \infty.
\end{array}\right.
	\]
Because of Theorem \ref{MPV_para} and the connectedness of $\Sigma$, $u$ (hence, $r$) is constant on $\Sigma$. This concludes the proof.
\end{proof}
\begin{proof}[of Corollary \ref{cor_spectrum}] Referring to the proof of Theorem \ref{omori-yau_vari} above, in our assumptions we can construct $u_\eps$ satisfying \eqref{prop_u_teo}. Then, by the chain rule, its restriction to $\Sigma$ satisfies
	\begin{equation}\label{wmpcor}
	\begin{array}{lll}
	\disp \frac{1}{\ell} \Delta_\Sigma u_\eps & \ge & \disp \frac{1}{\ell} \disp \mathrm{Tr}_{T\Sigma} (\nabla^2 u_\eps) + \langle {\bf H}, \nabla u_\eps \rangle \ge \PP_{\ell}^-[ \nabla^2 u_\eps] - |{\bf H}||\nabla u_\eps| \\[0.3cm]
	& \ge & \disp \bar \delta/2 \qquad \text{on } \, \{ u_\eps > \gamma\} \cap \Sigma. 
	\end{array}
	\end{equation}
Since $u_\eps^* : = \sup_{\Sigma} u_\eps < \infty$, we deduce that necessarily $\Sigma$ is stochastically incomplete (cf. Subsection \ref{rem_smooth}). To prove that $\sigma(\Delta_\Sigma)$ is discrete, recall that $u_\eps$ is given by approximating the radial function $u$  ensured by Proposition \ref{prop_constr_u}, say, $\|u_\eps - u\|_\infty < \eps$ on $\Omega_R$. The construction of $u$ and our assumption $\mathrm{dist}(x, \partial \Omega) \ra 0$ as $x \in \Sigma$, $x \to \infty$ imply that $u(x) \ra u^*$. Hence, the set $U_\eps = \{ x \in \Sigma : u(x) > u^* - \eps\}$ is an exterior region of $\Sigma$, namely, it has compact complement. Because of \eqref{wmpcor} and the maximum principle, $u_\eps^* = \sup_\Sigma u_\eps$ is not attained on $\Sigma$, and we can thus compute 
	\[
	\frac{- \Delta_\Sigma (u_\eps^*- u_\eps)}{u_\eps^*- u_\eps} \ge \frac{\ell \bar \delta}{2(2\eps + u^* - u)} \ge \frac{\ell \bar \delta}{6\eps} \qquad \text{on } \, U_\eps.
	\]
Hence, in view of Persson's formula and Barta's theorem (cf. \cite{bjmm} and Chapter 3 of \cite{bmr}) the infimum of the essential spectrum $\sigma_{\mathrm{ess}}(\Delta_\Sigma)$ is related to the bottom of the spectrum $\lambda_1^{\Delta_\Sigma}(U_\eps)$ of the exterior region $U_\eps$ as follows:
	\[
\inf \sigma_{\rm ess}(\Delta_\Sigma) = \lim_{\eps \ra 0} \lambda_1^{\Delta_\Sigma}(U_\eps) \ge \liminf_{\eps \ra 0} \inf_{U_\eps} \frac{- \Delta_\Sigma(u_\eps^*-u_\eps)}{u_\eps^*-u_\eps} \ge \liminf_{\eps \ra 0} \frac{\ell \bar \delta}{6\eps} = \infty,
	\]
hence $\Delta_\Sigma$ has discrete spectrum.	
\end{proof}

\vspace{6mm}

\flushleft
Jorge H. S. Lira\\
(corresponding author)\\
Departamento de Matem\'atica\\
Universidade Federal do Cear\'a\\
Campus do Pici - Bloco 914\\
60455-760 Fortaleza - Cear\'a, Brazil
jorge.lira@mat.ufc.br

\vspace{4mm}
Adriano A. Medeiros\\
Departamento de Matem\'atica\\
Universidade Federal da Para\'iba\\
58059-900, Jo\~ao Pessoa - Para\'iba, Brazil 
adriano@mat.ufc.br, adrianoalves@mat.ufpb.br

\vspace{4mm}
Luciano Mari\\
Dipartimento di Matematica ``G. Peano"\\
Universit\'a degli Studi di Torino\\
Via Carlo Alberto 10\\
10123 Torino - Italy\\
luciano.mari@unito.it

\vspace{4mm}
Eddygledson S. Gama\\
Departamento de Ci\^encia e Tecnologia\\
Universidade Federal Rural do Semi-\'Arido\\
Campus de Cara\'ubas, 59780-000 Cara\'ubas \\
Rio Grande do Norte - Brazil.\\
eddygledson@ufersa.edu.br

\end{document}